\documentclass[10pt]{amsart}

\usepackage{amsmath,amssymb,amsfonts,bbold,mathrsfs,amscd}
\usepackage{pinlabel,multirow}

\DeclareMathAlphabet\oldmathcal{OMS}        {cmsy}{b}{n}
\SetMathAlphabet    \oldmathcal{normal}{OMS}{cmsy}{m}{n}
\DeclareMathAlphabet\oldmathbcal{OMS}       {cmsy}{b}{n}

\usepackage{eucal}

\newtheorem{thm}{Theorem}
\newtheorem{prop}[thm]{Proposition}
\newtheorem{cor}[thm]{Corollary}
\newtheorem{lem}[thm]{Lemma}
\newtheorem{defn}[thm]{Definition}

\newenvironment{xpl}{\refstepcounter{thm} \medskip \noindent {\bf  Example \arabic{section}.\arabic{thm}}}{\hfill$\diamondsuit$\mbox{}\bigskip}

\newenvironment{rmk}{\refstepcounter{thm} \medskip \noindent {\bf  Remark \arabic{section}.\arabic{thm}.}}{\hfill\mbox{}\bigskip}

\newcounter{num}

\newenvironment{thmlist}{\begin{list}{(\roman{num})}{\usecounter{num}\setlength{\leftmargin}{25pt}
\setlength{\itemindent}{0pt}\setlength{\labelwidth}{20pt}\setlength{\labelsep}{5pt}\setlength{\itemsep}{0in}}}{\end{list}}

\newcommand{\Z}{\mathbb{Z}}
\newcommand{\Q}{\mathbb{Q}}
\newcommand{\R}{\mathbb{R}}
\newcommand{\C}{\mathbb{C}}

\newcommand{\re}{\operatorname{Re}}

\newcommand{\cps}{\mathbb{C}P}

\newcommand{\ol}[1]{\bar{#1}}

\newcommand{\Ad}{\operatorname{Ad}}

\newcommand{\Alb}{\operatorname{Alb}}
\newcommand{\Aut}{\operatorname{Aut}}
\newcommand{\aut}{\operatorname{\mathfrak{aut}}}
\newcommand{\hol}{\operatorname{\mathfrak{hol}}}
\newcommand{\Diff}{\operatorname{Diff}}
\newcommand{\Fol}{\operatorname{Fol}}
\newcommand{\fol}{\operatorname{\mathfrak{fol}}}
\newcommand{\CR}{\operatorname{CR}}
\newcommand{\Cr}{\operatorname{\mathfrak{cr}}}

\newcommand{\grad}{\operatorname{grad}}
\newcommand{\Hol}{\operatorname{Hol}}

\newcommand{\End}{\operatorname{End}}
\newcommand{\Inn}{\operatorname{Inn}}
\newcommand{\Isom}{\operatorname{Isom}}

\newcommand{\Ric}{\operatorname{Ric}}

\newcommand{\Vol}{\operatorname{Vol}}

\newcommand{\PSU}{\operatorname{PSU}}

\newcommand{\U}{\operatorname{U}}

\newcommand{\SO}{\operatorname{SO}}
\newcommand{\Sp}{\operatorname{Sp}}
\newcommand{\G}{\operatorname{G}}

\newcommand{\sm}[1]{\scriptscriptstyle{#1}}
\newcommand{\Sm}[1]{\scriptstyle{#1}}
\newcommand{\contr}{\,\lrcorner\,}
\newcommand{\CX}{\mbox{${\mathcal X}\hspace{-.8em}-\,$}}
\newcommand{\simarrow}{\stackrel{\sim}{\longrightarrow}}
\newcommand{\superscript}[1]{\ensuremath{^{\textrm{#1}}}}

\title{Stability of Sasaki-extremal metrics under complex deformations}
\author{Craig van Coevering}
\address{Max-Planck-Institut f\"{u}r Mathematik, Vivatsgasse 7, 53111 Bonn Germany}
\email{craigvan@mpim-bonn.mpg.de}
\date{April 5, 2012}

\begin{document}

\begin{abstract}

We consider the stability of Sasaki-extremal metrics under deformations of the transversal complex structure of the Sasaki
foliation $\mathscr{F}_{\xi}$, induced by the Reeb vector field $\xi$.  Let $g$ be a Sasaki-extremal metric on $M$, $G$ a compact connected subgroup of the automorphism group of the Sasaki structure, and suppose the reduced scalar curvature satisfies
$s_g^G =0$.  And consider a $G$-equivariant deformation $(\mathscr{F}_\xi, \ol{J}_t)_{t\in\mathcal{B}}$ of
of the transversely holomorphic foliation preserving
$\mathscr{F}_{\xi}$ as a smooth foliation.  Provided the Futaki invariant relative to $G$ of $g$ is nondegenerate, 
the existence of Sasaki-extremal metrics is preserved under small variations of $t\in\mathcal{B}$ and of the Reeb vector
$\xi\in\mathfrak{z}$ in the center of $\mathfrak{g}$.  If $G=T\subseteq\Aut(g,\xi)$ is a maximal torus, the nondegeneracy of the Futaki invariant is automatic.  So such deformations provide the easiest examples.

When the initial metric $g$ is Sasaki-Einstein this result can be improved using known properties of the Futaki invariant 
Although the relative Futaki invariant is useless in this case, non-trivial deformations can be obtained when
$G=T\subseteq\Aut(g,\xi)$ is a maximal torus.  Then a slice of the above family of Sasaki-extremal metrics is Sasaki-Einstein.
Thus for each $t\in\mathcal{B}$ there is a $\xi_t\in\mathfrak{z}$ so that the Sasaki-extremal metric with
Reeb vector field $\xi_t$ is Sasaki-Einstein.  We apply this to deformations of toric 3-Sasaki 7-manifolds to obtain
new families of Sasaki-Einstein metrics on these manifolds, which are deformations of 3-Sasaki metrics.

\end{abstract}

\maketitle

\section{Introduction}

Recall that a polarization on a complex manifold $M$ is and element $\Omega\in H^{1,1}(M)\cap H^2(M,\R)$ such that $\Omega$ can
be represented by a K\"{a}hler form $\omega_g$ of a K\"{a}hler metric $g$ on $M$.  In the hope of finding a canonical metric in
the polarization E. Calabi~\cite{Cal82,Cal85} defined a natural Riemannian functional on this space of K\"{a}hler metrics.  Denote by
$\mathfrak{M}_\Omega$ the space of K\"{a}hler metrics representing the polarization.   Calabi proposed that one should seek
critical points of the functional
\begin{equation}\label{eq:Calabi-funct}
\begin{array}{rcl}
\mathfrak{M}_\Omega & \overset{\mathcal{C}}{\longrightarrow} & \R \\
g & \mapsto & \int_M s_g^2 \, d\mu_g
\end{array}
\end{equation}
where $s_g$ is the scalar curvature of $g$ and $d\mu_g$ the volume form.  He called the critical points of this functional
\emph{extremal K\"{a}hler metrics} and showed that $g$ is extremal if and only if the gradient of $s_g$ is a real holomorphic
vector field.  In particular, a constant scalar curvature metric is extremal, but many examples of extremal metrics are known
which are not constant scalar curvature.  An extremal K\"{a}hler metric is of constant scalar curvature if and only if the
Futaki invariant vanishes~\cite{Fut83,Cal85}.  Many examples are known of extremal metrics of both constant and nonconstant scalar curvature

One way of producing new examples is to start with a known extremal metric and try to deform the solution as either the K\"{a}hler
class or complex structure varies.  This has been done with considerable success by C. LeBrun and S. R. Simanca~\cite{LebSim94,LebSim93}, where it was shown that there is no obstruction to deforming extremal metrics as the K\"{a}hler class is varied, whereas a nondegeneracy condition on the Futaki invariant is sufficient to guarantee that a constant scalar curvature metric can be deformed through extremal metrics as the complex structure is deformed.  The nondegeneracy of the Futaki invariant is necessary as deforming the complex structure can result in a reduction of the size of the automorphism group.  Later Y. Rollin, S. R. Simanca, and C. Tipler~\cite{RolSimTip11} generalized the later result to the case of equivariant deformations of the complex structure, where the sufficient condition becomes the nondegeneracy of the relative Futaki invariant.

This article gives analogous results for Sasaki manifolds.  Similar results as in~\cite{RolSimTip11} are proved, although the a polarization of a Sasaki manifold is given by a choice of Reeb vector field, rather than a K\"{a}hler class.  Thus the notions
of the Sasaki polarization and nondegeneracy of the relative Futaki invariant involve varying the Reeb vector field.

\subsection{Main result}

Sasaki geometry sits between two K\"{a}hler geometries.  If $(M,g)$ has is Sasaki then the metric cone
$(C(M)=\R_{>0}\times M,\ol{g}=dr^2 +r^2 g)$ is K\"{a}hler for some almost complex structure.  Furthermore, a Sasaki structure is contact and
the foliation $\mathscr{F}_{\xi}$ generated by the Reeb vector field $\xi$ is transversely K\"{a}hler, i.e. locally the transversal space to the leaves has a complex structure $\ol{J}$ so that the induced metric $g^T$ is K\"{a}hler.  Alternatively,
$\ol{J}$ is an integrable almost complex structure on $\nu(\mathscr{F}_{\xi})=TM/{\tau(\mathscr{F}_{\xi})}$, where $\tau(\mathscr{F}_{\xi})\subset TM$ is the subbundle tangent to the leaves.

So it is not surprising that the notion of extremal metric can be defined analogously for Sasaki metrics using the same functional (\ref{eq:Calabi-funct}) defined on the space $\mathfrak{M}(\xi,\ol{J})$ of metrics arising from Sasaki structures with Reeb
vector field $\xi$ and transversal complex structure $\ol{J}$.  This program was carried out in~\cite{BoyGalSim08}.  See also~\cite{BoyGalSim09}.
Not surprisingly, critical points are Sasaki metrics $g$ with the gradient of $s_g$ a transversally real holomorphic vector field.
One notable difference from the K\"{a}hler case is the role of the polarization $\Omega$ is taken by the Reeb vector field
$\xi$.  The stability of extremal solutions under variations of $\xi$ was proved by C. P. Boyer, K. Galicki, and
S. R. Simanaca~\cite{BoyGalSim08}.

The goal of this article is to give a similar stability result for Sasaki-extremal
metrics under equivariant deformations of the transversal complex structure to the Reeb foliation.  The results we obtain are similar to those in~\cite{RolSimTip11} in the K\"{a}hler case.  Let $(g,\eta,\xi,\Phi)$ is a Sasaki-extremal structure on $M$.  Then as in the K\"{a}hler case~\cite{Cal85}, it was shown in~\cite{BoyGalSim08} that the identity component of automorphism group of the Sasaki structure
$\Aut(g,\eta,\xi,\Phi)_0$ is a maximal compact subgroup of $\Fol(M,\mathscr{F}_{\xi},\ol{J})_0$, the identity component of the group
of transversely holomorphic automorphisms of the foliation $\mathscr{F}_{\xi}$.  Let $G\subseteq G'=\Aut(g,\eta,\xi,\Phi)_0$
be a connected subgroup with Lie algebras $\mathfrak{g}\subseteq\mathfrak{g}'$ so that $\xi\in\mathfrak{g}$.
Then $\mathfrak{g}'/\{\xi\} \subseteq\mathfrak{h}^T(\xi,\ol{J})_0$, where $\mathfrak{h}^T(\xi,\ol{J})_0$ is the subspace of
transversely holomorphic vector fields modulo those tangent to the leaves, in $\tau(\mathscr{F}_{\xi})$, that have
\emph{holomorphy potentials}, i.e. are of the form $\partial^{\#}\phi :=(\ol{\partial}\phi)^\#$ for a \emph{basic} function
$\phi$.  Let $\mathfrak{z}=Z(\mathfrak{g})$ be the center of $\mathfrak{g}$ and $\mathfrak{z}'=C_{\mathfrak{g}'}(\mathfrak{g})$ the centralizer of $\mathfrak{g}$ in $\mathfrak{g}'$.  Also define $\mathfrak{p}=N_{\mathfrak{g}'}(\mathfrak{g})$ to be the normalizer
of $\mathfrak{g}$ in $\mathfrak{g}'$.  It will turn out that $\mathfrak{p}/\mathfrak{g} =\mathfrak{z}'/\mathfrak{z}$.

Denote the space of $G$-invariant smooth functions by $C^\infty(M)^G$.  A \emph{transversal deformation} of the Sasaki structure
$(g,\eta,\xi,\Phi)$ is a Sasaki structure $(\tilde{g},\tilde{\eta},\xi,\tilde{\Phi})$ with
$\tilde{\eta} =\eta +d^c \phi$ for $\phi\in C^\infty(M)^G$ and transversal K\"{a}hler form
$\tilde{\omega}^T =\omega^T +\frac{1}{2}dd^c \phi$, while the Reeb vector field and transversal complex structure $\ol{J}$ remain unchanged.  We can then introduce the notion of the reduced scalar curvature $s_g^G$ for any
$G$-invariant Sasaki structure, and the Futaki invariant relative to $G$
\[ \mathcal{F}_{G,\xi} : \mathfrak{p}/\mathfrak{g} \rightarrow\R, \]
where $\mathfrak{p}$ is the normalizer of $\mathfrak{g}$ in $\mathfrak{g}'$, which is independent of the Sasaki structure with
Reeb vector field $\xi$ and transversal complex structure $\ol{J}$.  This space of Sasaki structure we denote
$\mathcal{S}(\xi,\ol{J})$.  On the space of $G$-invariant structures  $\mathcal{S}(\xi,\ol{J})^G$ the condition $s_g^G =0$
is equivalent to $g$ being Sasaki-extremal and $\mathcal{F}_{G,\xi} \equiv 0$.

The connected component of the identity of the center is a torus $T^r\subseteq G$, and the contact structure defines a moment map
on the cone $C(M) =\R_{>0}\times M$,
\begin{equation} \label{eq:moment-map}
 \mu_{\eta} :C(M) \rightarrow\mathfrak{z}^*,
\end{equation}
where $\mu_{\eta}(x,r)(X) = r^2\eta_x(X_x)$, with $X\in\mathfrak{z}$ and $X_x$ the induced vector at $x\in M$.
The image of (\ref{eq:moment-map}) is a convex polyhedral cone
$\mathcal{C}_{\mathfrak{z}}^*\subset\mathfrak{z}^*$ (\cite{MorTom97}).  Although $\mu_{\eta}$ obvious depends on $\eta$, the image $\mathcal{C}_{\mathfrak{z}}^*$ is the same for any transversal deformation $\tilde{\eta} =\eta +d^c \phi$, for
$\phi\in C^\infty(M)_G$, and turns out to be independent of a choice of Reeb vector fields in $\mathfrak{z}$.
Define $\mathfrak{z}^+ =\{\zeta\in\mathfrak{z}\ :\ \eta(\zeta)>0 \}$.  If $\zeta\in\mathfrak{z}^+$, then
$\eta_{\zeta} =\eta(\zeta)^{-1}\eta$ is easily seen to be a contact form of a Sasaki structure with the same
CR structure as $(\xi,\tilde{\eta},\tilde{\Phi},\tilde{g})$ and with Reeb vector field $\zeta$.  Fakas' theorem says the dual
cone $\mathcal{C}_{\mathfrak{z}}$ to $\mathcal{C}_{\mathfrak{z}}^*$ is a convex polyhedral cone, and from
(\ref{eq:moment-map}) we see that
\[ \mathfrak{z}^+ =\overset{\circ}{\mathcal{C}}_{\mathfrak{z}}.\]

Differentiating the relative Futaki invariant with $\xi$ varying in $\mathfrak{z}$ induces a linear map
\begin{equation}\label{eq:Futaki-dif}
 \mathfrak{p}/\mathfrak{g} \rightarrow\mathfrak{z}^*,
\end{equation}
and we say that the Futaki invariant $\mathcal{F}_{G,\xi}$ relative to $G$ is nondegenerate if it is injective.

We consider $G$-equivariant deformations of the transversal complex structure of the Reeb foliation $(\mathscr{F}_{\xi}, \ol{J})$.
We fix the smooth structure of $\mathscr{F}_{\xi}$, so a deformation is given by $(\mathscr{F}_{\xi}, \ol{J}_t)_{t\in\mathcal{B}}$.
The holomorphic structure on $\mathscr{F}_{\xi}$ has a versal deformation space~\cite{KacNic89,Gir92}, with tangent space \[H^1_{\ol{\partial}_b}(\mathcal{A}^{0,\bullet}),\quad\text{where } \mathcal{A}^{0,k} =\Gamma\Bigl(\Lambda^{0,k}_b\otimes\nu(\mathscr{F}_{\xi})^{1,0}\Bigr)\]
denotes the basic $(0,k)$-forms with values in $\nu(\mathscr{F}_{\xi})^{1,0}$ and
\[  0\longrightarrow \mathcal{A}^{0,0} \overset{\ol{\partial}_b}{\longrightarrow}\mathcal{A}^{0,1} \overset{\ol{\partial}_b}{\longrightarrow}\cdots \]
is the \emph{basic} Dolbeault complex with values in the transverse holomorphic bundle $\nu(\mathscr{F}_{\xi})^{1,0}$.
Then $H^1_{\ol{\partial}_b}(\mathcal{A}^{0,\bullet})^G$ is the tangent space to the $G$-equivariant deformations of
$(\mathscr{F}_{\xi}, \ol{J})$.  By~\cite{KacGmi97} the transversely K\"{a}hler property of $(\mathscr{F}_{\xi}, \ol{J})$ is stable
under small deformations.  But the existence of a compatible Sasaki structure may be obstructed.
The obstruction, due to H. Nozawa~\cite{Noz11}, is reviewed in Section~\ref{subsect:defor}.  An unobstructed deformation
$(\mathscr{F}_{\xi}, \ol{J}_t)_{t\in\mathcal{B}}$ is said to be of \emph{$(1,1)$-type}.  If $\mathcal{B}$ is smooth, after 
possibly restricting to a neighborhood of zero in $\mathcal{B}$, there is a family 
$(g_t,\eta_t,\xi,\Phi_t)\in\mathcal{S}(\xi,\ol{J}_t),\ t\in\mathcal{B}$.  And if $(\mathscr{F}_{\xi}, \ol{J}_t)_{t\in\mathcal{B}}$
is $G$-equivariant, we may assume that the family $(g_t,\eta_t,\xi,\Phi_t),\ t\in\mathcal{B}$ is $G$-equivariant.
In particular, if $\Ric_g >0$ then the obstruction vanishes on a neighborhood of zero in any deformation. 

Associated to the family $(\mathscr{F}_{\xi}, \ol{J}_t)_{t\in\mathcal{B}}$ for sufficiently small $\phi\in C^\infty(M)^G$ we consider the Sasaki metrics $g_{t,\alpha,\phi}$ with transverse K\"{a}hler form
\[\omega^T_{t,\alpha,\phi} = \omega^T_{t,\alpha} +\frac{1}{2}dd^c \phi, \]
with Reeb vector field $\xi+\alpha\in\mathfrak{z}^+,\ \eta_{t,\alpha,\phi} =\eta_{t,\xi+\alpha} +d^c \phi$ and
$\omega^T_{t,\alpha} =\frac{1}{2}d\eta_{t,\xi+\alpha}$.  Thus we have a family of Sasaki
metrics parametrized by a neighborhood of $(0,0,0)\in\mathcal{B}\times\mathfrak{z} \times C^\infty(M)^G$.  Assuming that
$g_{0,0,0}$ is Sasaki-extremal and satisfies $s^G_{g_{0,0,0}} =0$, we seek solutions to $s^G_{g_{t,\alpha,\phi}} =0$ for
$(t,\alpha,\phi)$ close to zero.  Using suitable Banach spaces, an application of the implicit function theorem gives the main
theorem.
\begin{thm}\label{thm:main}
Let $(g,\xi,\eta,\Phi)$ be a Sasaki-extremal structure satisfying $s_g^G =0$.  Suppose $G\subseteq G'=\Aut(\xi,\eta,\Phi,g)_0$
is a connected compact subgroup and $(\mathscr{F}_{\xi}, \ol{J}_t)_{t\in\mathcal{B}}$ a $G$-equivariant deformation of $(1,1)$-type.  
If the Futaki invariant relative to $G$ is nondegenerate $g$, then there is a small neighborhood of zero $W\subset\mathcal{B}\times\mathfrak{z}\times C^\infty(M)^G$ and a smooth closed submanifold $V\subset W$, with $\dim V =\dim_{\R}\mathcal{B}+\dim\mathfrak{z}$ so that for $(t,\alpha,\phi)\in V$ there is Sasaki metric $g_{t,\alpha,\phi}$ satisfying $s^G_{g_{t,\alpha,\phi}} =0$.  Therefore, there is a space of Sasaki-extremal metrics parametrized by $V$.
Furthermore, the projection $\pi:V\rightarrow\mathcal{B}$ is a submersion with fibers of dimension $\dim\mathfrak{z}$.
\end{thm}

Unfortunately, the nondegeneracy of the relative Futaki invariant is not an easy condition to work with, and from
(\ref{eq:Futaki-dif}) we see that $\mathfrak{z}$ must be sufficiently large in $\mathfrak{g}$ for it to hold.  Fortunately,
if $G= T\subset G'$ is a maximal torus, then the relative Futaki invariant is trivially nondegenerate as $\mathfrak{p}/\mathfrak{g} =0$.
\begin{cor}\label{cor:main-tor}
Let $(g,\eta,\xi,\Phi)$ be a Sasaki-extremal structure satisfying $s_g^G =0$.  Suppose that
$(\mathscr{F}_{\xi}, \ol{J}_t)_{t\in\mathcal{B}}$ is a
$G$-equivariant deformation of $(1,1)$-type, where $G\subseteq G'=\Aut(g,\eta,\xi,\Phi)_0$ is a maximal torus.  
Then there is a neighborhood of zero
$W\subset\mathfrak{B}\times\mathfrak{g}$, so that for $(t,\alpha)\in W$ there is Sasaki metric $g_{t,\alpha,\phi_{t,\alpha}}$ satisfying $s^G_{g_{t,\alpha,\phi_{t,\alpha}}} =0$.  So for each fixed $t\in\mathcal{B}$
close to zero, the space of extremal metrics is locally parametrized by a neighborhood of zero in $\mathfrak{g}$.
\end{cor}

Just as in the K\"{a}hler-Einstein case~\cite{LebSim94} the linear map (\ref{eq:Futaki-dif}) is always trivial when $(g,\eta,\xi,\Phi)$ is Sasaki-Einstein.   Fortunately, Corollary~\ref{cor:main-tor} is still useful in obtaining new examples of Sasaki-Einstein metrics
when $G$ is a maximal torus.
In this case one can use the nondegeneracy of the Futaki invariant on $\mathfrak{g}$~\cite{MarSpaYau08,FutOnWan09} to show that
there is a neighborhood $W\subset\mathcal{B}\times\mathfrak{g}$ so that for each $(t,0)\in W$ there is a $\alpha_t\in\mathfrak{g}$ so that $g_{t,\alpha_t,\phi_t}$ is Sasaki-Einstein.
\begin{cor}\label{cor:main-Einst}
Let $(g,\eta,\xi,\Phi)$ be a Sasaki-Einstein structure, and suppose that $(\mathscr{F}_{\xi}, \ol{J}_t)_{t\in\mathcal{B}}$ is a
$G$-equivariant deformation, where $G\subseteq G'=\Aut(g,\eta,\xi,\Phi)_0$ is a maximal torus. Then there is a neighborhood
$U\subset\mathcal{B}$ so that for $t\in U$ there is a unique $\alpha_t \in\mathfrak{g}$ and a $\phi_t \in C^\infty(M)^G$ so that $g_{t,\alpha_t,\phi_t}$ is Sasaki-Einstein.
\end{cor}

In the final section Corollary~\ref{cor:main-Einst} is applied to give a new family of Sasaki-Einstein metrics by deforming
the underlying Sasaki-Einstein metric on toric 3-Sasaki 7-manifolds.  These manifolds were studied in~\cite{BGMR98} as torus 3-Sasaki
quotients of spheres where it was proved that there are infinitely many of each Betti number $b_2 (M)\geq 1$.  Note that
they are not toric as Sasaki-Einstein manifolds.  If $b_2(M)\geq 2$, then for a fixed Sasaki structure
$G=T^3=\Aut(g,\xi,\eta,\Phi)_0$ is a 3-torus.
It was calculated by the author in~\cite{vanC12} that
$H^1_{\ol{\partial}_b}(\mathcal{A}^{0,\bullet})=H^1_{\ol{\partial}_b}(\mathcal{A}^{0,\bullet})^{T^3} = b_2 (M) -1$ giving a smooth
versal deformation space.  Thus Corollary~\ref{cor:main-tor} gives a neighborhood $W\subset\C^{b_2(M)-1}\times\mathfrak{g}$ of zero
parametrizing a space of Sasaki-extremal metrics.  And furthermore, Corollary~\ref{cor:main-Einst} gives
a slice of Sasaki-Einstein metrics.  There is a neighborhood $U\subset\C^{b_2(M)-1}$ so that for $t\in U$ there is
an $\alpha_t \in\mathfrak{g}$ and $\phi_t\in C^\infty(M)^G$ so that $g_{t,\alpha_t,\phi_t}$ is Sasaki-Einstein and contained in a real 3-dimensional space of Sasaki-extremal metrics.  Note that the Einstein metrics in this family have three different
isometry groups as shown in Figure~\ref{fig:def-space}.

These examples provide the first examples, to the author's knowledge, of deformations of 3-Sasaki metrics to metrics which are
Sasaki-Einstein but not 3-Sasaki.  These are also examples of Einstein manifolds admitting 3 Killing spinors with deformations to 
Einstein metrics admitting only 2 Killing spinors.  More details will appear in~\cite{vanC12}.

\subsection{Acknowledgements}
I would like to thank the Max Planck Institute for Mathematics for their hospitality and excellent research environment
that I enjoyed while writing this article.

\section{Background}

\subsection{Sasaki manifolds}

\begin{defn}
A Riemannian manifold $(M,g)$ is a \emph{Sasaki manifold}, or has a compatible Sasaki structure, if the metric cone
$(C(M),\ol{g})=(\R_{>0} \times M, dr^2 +r^2 g)$ is K\"{a}hler with respect to some complex structure $I$, where $r$ is the
usual coordinate on $\R_{>0}$.
\end{defn}
Thus $M$ is odd and denoted $n=2m+1$, while $C(M)$ is a complex manifold with $\dim_{\C} C(M) =m+1$.

Although, this is the simplest definition, Sasaki manifolds were originally defined as a special type of metric contact
structure.  See the monograph~\cite{BoyGal08} or~\cite{FutOnWan09} for more details on the properties of Sasaki manifolds that we
summarize below.  We will identify $M$ with the $\{1\}\times M\subset C(M)$.  Let $r\partial_r$ be the Euler vector field on $C(M)$,
then it is easy to see that $\xi =Ir\partial_r$ is tangent to $M$.  Using the warped product formulae for the cone metric
$\ol{g}$~\cite{ONeil83} it is easy check that $r\partial_r$ is real holomorphic, $\xi$ is Killing with respect to both $g$ and
$\ol{g}$, and furthermore the orbits of $\xi$ are geodesics on $(M,g)$.
Define $\eta =\frac{1}{r^2}\xi\contr\ol{g}$, then we have
\begin{equation}
\eta =-\frac{I^* dr}{r} =d^c \log r,
\end{equation}
where $d^c =\sqrt{-1}(\ol{\partial} -\partial)$.  If $\omega$ is the K\"{a}hler form of $\ol{g}$, i.e.
$\omega(X,Y) =\ol{g}(IX,Y)$, then $\mathcal{L}_{r\partial_r} \omega =2\omega$ which implies that
\begin{equation}\label{eq:Kaehler-pot1}
\omega =\frac{1}{2}d(r\partial_r \contr\omega) =\frac{1}{2}d(r^2 \eta)=\frac{1}{4}dd^c(r^2).
\end{equation}
From (\ref{eq:Kaehler-pot1}) we have
\begin{equation}\label{eq:Kaehler-pot2}
\omega=rdr\wedge\eta +\frac{1}{2}r^2 d\eta.
\end{equation}

We will use the same notation to denote $\eta$ and $\xi$ restricted to $M$.  Then (\ref{eq:Kaehler-pot2}) implies that
$\eta$ is a contact form with Reeb vector field $\xi$, since $\eta(\xi)=1$ and $\mathcal{L}_{\xi} \eta =0$.
Let $D\subset TM$ be the contact distribution which is defined by
\begin{equation}
D_x =\ker\eta_x
\end{equation}
for $x\in M$.  Furthermore, if we restrict the almost complex structure to $D$, $J:=I|_D$, then $(D,J)$ is a strictly pseudoconvex CR structure on $M$.  We have a splitting of the tangent bundle $TM$
\begin{equation}
TM =D\oplus L_{\xi},
\end{equation}
where $L_{\xi}$ is the trivial subbundle generated by $\xi$.  It will be convenient to define a tensor $\Phi\in\End(TM)$ by
$\Phi|_D =J$ and $\Phi(\xi) =0$.  Then
\begin{equation}\label{eq:comp-tens}
\Phi^2 =-\mathbb{1} +\eta\otimes\xi.
\end{equation}
Since $\xi$ is Killing, we have
\begin{equation}
d\eta (X,Y) =2 g(\Phi(X),Y), \quad\text{where }X,Y\in\Gamma(TM),
\end{equation}
and $\Phi(X) =\nabla_X \xi$, where $\nabla$ is the Levi-Civita connection of $g$.  Making use of (\ref{eq:comp-tens}) we see
that
\[ g(\Phi X,\Phi Y) =g(X,Y) -\eta(X)\eta(Y), \]
and one can express the metric by
\begin{equation}\label{eq:metric}
g(X,Y)=\frac{1}{2}(d\eta)(X,\Phi Y) +\eta(X)\eta(Y).
\end{equation}

We will denote a Sasaki structure on $M$ by $(g,\eta,\xi,\Phi)$.  Although, the reader can check that merely specifying
$(g,\xi),\ (g,\eta),$ or $(\eta, \Phi)$ is enough to determine the Sasaki structure, it will be convenient to denote the
remaining structure.

The \emph{Reeb} foliation $\mathscr{F}_\xi$ on $M$ generated by the action of $\xi$ will be important in the sequel.
Note that it has geodesic leaves and is a Riemannian foliation, that is has a $\xi$ invariant Riemannian metric on the
normal bundle $\nu(\mathscr{F}_\xi)$.  But in general the leaves are not compact.  If the leaves are compact, or equivalently
$\xi$ generates an $S^1$-action, then $(g,\eta,\xi,\Phi)$ is said to be a \emph{quasi-regular} Sasaki structure, otherwise it is
irregular.  If this $S^1$ action is free, then $(g,\eta,\xi,\Phi)$ is said to be \emph{regular}.  In this last case
$M$ is an $S^1$-bundle over a manifold $Z$, which we will see below is K\"{a}hler.  If the structure if merely quasi-regular, then
the leaf space has the structure of a K\"{a}hler orbifold $Z$.
In general, in the irregular case, the leaf space is not even Hausdorff but we will make use of the transversally K\"{a}hler
property of the foliation $\mathscr{F}_\xi$ which we discuss next.

\subsection{Transverse K\"{a}hler structure}

We now describe a transverse K\"{a}hler structure on $\mathscr{F}_{\xi}$.
The vector field $\xi -\sqrt{-1}I\xi =\xi +\sqrt{-1}r\partial_r$ is holomorphic on $C(M)$.  If we denote by $\tilde{\C}^*$ the
universal cover of $\C^*$, then $\xi +\sqrt{-1}r\partial_r$ induces a holomorphic action
of $\tilde{\C}^*$ on $C(M)$.  The orbits of $\tilde{\C}^*$ intersect $M\subset C(M)$ in the orbits of the Reeb foliation
generated by $\xi$.  We denote the Reeb foliation by $\mathscr{F}_\xi$.  This gives $\mathscr{F}_\xi$ a transversely holomorphic
structure.

The foliation $\mathscr{F}_{\xi}$ together with its transverse holomorphic structure is given by an open covering
$\{U_\alpha \}_{\alpha\in A}$ and submersions $\pi_\alpha :U_\alpha \rightarrow W_\alpha \subset\C^m$ such that
when $U_\alpha \cap U_\beta \neq\emptyset$ the map
\[\phi_{\beta\alpha} =\pi_\beta \circ\pi_\alpha^{-1} :\pi_{\alpha}(U_\alpha \cap U_\beta) \rightarrow\pi_{\beta}(U_\alpha \cap U_\beta) \]
is a biholomorphism.

Not that on $U_\alpha$ the differential $d\pi_\alpha :D_x \rightarrow T_{\pi_\alpha(x)}W_\alpha$ at $x\in U_\alpha$ is
an isomorphism taking the almost complex structure $J_x$ to that on $T_{\pi_\alpha(x)}W_\alpha$.
Since $\xi\contr d\eta =0$ the 2-form $\frac{1}{2}d\eta$ descends to a form $\omega_\alpha^T$ on $W_\alpha$.  Similarly,
$g^T =\frac{1}{2}d\eta(\cdot,\Phi\cdot)$ satisfies $\mathcal{L}_\xi g^T =0$ and vanishes on vectors tangent to the leaves, so
it descends to an Hermitian metric $g^T_\alpha$ on $W_\alpha$ with K\"{a}hler form $\omega_\alpha^T$.  The K\"{a}hler metrics
$\{g_\alpha ^T \}$ and K\"{a}hler forms $\{\omega_\alpha^T \}$ on $\{ W_\alpha\}$ by construction are isomorphic on the overlaps
\[ \phi_{\beta\alpha} : \pi_{\alpha}(U_\alpha \cap U_\beta) \rightarrow\pi_{\beta}(U_\alpha \cap U_\beta).\]
We will use $g^T$, respectively $\omega^T$, to denote both the K\"{a}hler metric, respectively K\"{a}hler form, on the the
local charts and the globally defined pull-back on $M$.

If we define $\nu(\mathscr{F}_\xi) =TM/{L_\xi}$ to be the normal bundle to the leaves, then we can generalize the above concept.
A tensor $\Psi\in\Gamma\bigl((\nu(\mathscr{F}_\xi)^*)^{\otimes p} \bigotimes\nu(\mathscr{F}_\xi)^{\otimes q}\bigr)$ is \emph{basic} if
$\mathcal{L}_V \Psi =0$ for any vector field $V\in\Gamma(L_\xi)$.  It is sufficient to check this for $V=\xi$.
Then $g^T$ and $\omega^T$ are such tensors on $\nu(\mathscr{F}_\xi)$.  We will also make use of the bundle isomorphism
$\pi:D \rightarrow\nu(\mathscr{F}_\xi)$, which induces an almost complex structure $\ol{J}$ on $\nu(\mathscr{F}_\xi)$ so that
$(D,J)\cong(\nu(\mathscr{F}_\xi),\ol{J})$ as complex vector bundles.  Clearly, $\ol{J}$ is basic and is mapped to the
almost complex structure by local charts $d\pi_\alpha :D_x \rightarrow T_{\pi_\alpha(x)}W_\alpha$.

To work on the K\"{a}hler leaf space we define the Levi-Civita connection of $g^T$ by
\begin{equation}
\nabla^T_X Y =\begin{cases}
\pi_\xi(\nabla_X Y) & \text{ if }X, Y\text{ are smooth sections of }D, \\
\pi_\xi([V,Y]) & \text{ if } X=V\text{ is a smooth section of }L_\xi,
\end{cases}
\end{equation}
where $\pi_\xi :TM \rightarrow D$ is the orthogonal projection onto $D$.  Then $\nabla^T$ is the unique torsion free connection
on $D\cong\nu(\mathscr{F}_\xi)$ so that $\nabla^T g^T=0$.  Then for $X,Y\in\Gamma(TM)$ and $Z\in\Gamma(D)$ we have the
curvature of the transverse K\"{a}hler structure
\begin{equation}
R^T(X,Y)Z =\nabla^T_X \nabla^T_Y Z -\nabla^T_Y \nabla^T_X Z -\nabla^T_{[X,Y]} Z,
\end{equation}
and similarly we have the transverse Ricci curvature $\Ric^T$ and scalar curvature $s^T$.  We will denote the
transverse Ricci form by $\rho^T$.

The following follows from O'Neill tensor computations for a Riemannian submersion.  See~\cite{ONeil66} and~\cite[Ch. 9]{Bess87}.
\begin{prop}\label{prop:Sasaki-Ric}
Let $(M,g,\eta,\xi,\Phi)$ be a Sasaki manifold of dimension $n=2m+1$, then
\begin{thmlist}
\item  $\Ric_g (X,\xi) =2m\eta(X),\quad\text{for }X\in\Gamma(TM)$,\label{eq:submer-Ric-Reeb}
\item  $\Ric^T (X,Y) =\Ric_g (X,Y) +2g^T(X,Y),\quad\text{for }X,Y\in\Gamma(D),$
\item  $s^T =s_g +2m.$\label{eq:submer-scal}
\end{thmlist}
\end{prop}
\begin{defn}
A \emph{Sasaki-Einstein} manifold $(M,g,\eta,\xi,\Phi)$ is a Sasaki manifold with
\[ \Ric_g =2m\, g.\]
\end{defn}
Note that by (\ref{eq:submer-Ric-Reeb}) the Einstein constant must be $2m$, and the transverse K\"{a}hler metric is
also Einstein
\begin{equation}
\Ric^T =(2m+2)\, g^T.
\end{equation}
Conversely, if one has a Sasaki structure $(g,\eta,\xi,\Phi)$ with $\Ric^T =\tau\, g^T$ with $\tau>0$, then after a
$D$-homothetic transformation one has a Sasaki-Einstein structure $(g',\eta',\xi',\Phi)$, where 
$\eta'=a\eta,\ \xi'=a^{-1}\xi$, and $g'=ag +a(a-1)\eta\otimes\eta$, with $a=\frac{\tau}{2m+2}$.

Let $\mathcal{S}(\xi)$ be the space of Sasaki structures $(\tilde{g},\tilde{\eta},\tilde{\xi},\tilde{\Phi})$ on $M$ with
$\tilde{\xi}=\xi$.  For any $(\tilde{g},\tilde{\eta},\tilde{\xi},\tilde{\Phi})\in\mathcal{S}(\xi)$ the 1-form
$\beta=\tilde{\eta}-\eta$ is basic, so $[d\tilde{\eta}]_b=[d\eta]_b$, where $[\,\cdot\,]_b$ denotes the basic cohomology
class of a basic closed form.  Thus $[\omega^T]_b \in H^2_b(M/\mathscr{F}_\xi,\R)$ is the same for every 
Sasaki structure in $\mathcal{S}(\xi)$.
Thus, as first observed in~\cite{BoyGalSim08}, fixing the Reeb vector field is the closest analogue to a polarization in K\"{a}hler geometry, and we say that the Reeb vector field $\xi$ \emph{polarizes} the Sasaki manifold.

For a fixed Reeb vector field $\xi$, we consider a fixed transversal complex structure on $\mathscr{F}_\xi$ which is
equivalent to fixing $\ol{J}$ on $\nu(\mathscr{F}_\xi)$.  We define $\mathcal{S}(\xi,\ol{J})\subset\mathcal{S}(\xi)$
to be the subset of Sasaki structures inducing the same complex normal bundle $(\nu(\mathscr{F}_\xi),\ol{J})$, 
in other words, the set of $(\tilde{g},\tilde{\eta},\xi,\tilde{\Phi})\in\mathcal{S}(\xi)$ such that the following diagram
commutes
\begin{equation}\label{eq:trans-CD}
\begin{CD}
TM @>{\tilde{\Phi}}>> TM \\
@VVV            @VVV \\
\nu(\mathscr{F}_\xi) @>{\ol{J}}>> \nu(\mathscr{F}_\xi).
\end{CD}
\end{equation}

We will consider three different deformations of a Sasaki structure.  First we consider \emph{transverse K\"{a}hler}
deformations.
\begin{lem}[\cite{BoyGal08,BoyGalSim08}]\label{lem:trans-def}
The space $\mathcal{S}(\xi,\ol{J})$ of all Sasaki structures with Reeb vector field $\xi$ and transverse holomorphic
structure $\ol{J}$  is an affine space modeled on $C^\infty_b (M)/\R \times C^\infty_b (M)/\R \times H^1(M,\R)$.
If $(g,\eta,\xi,\Phi)\in\mathcal{S}(\xi,\ol{J})$ is a fixed Sasaki structure then another structure
$(\tilde{g},\tilde{\eta},\tilde{\xi},\tilde{\Phi})\in\mathcal{S}(\xi,\ol{J})$ is determined by real basic functions
$\phi$ and $\psi$ and an harmonic, with respect to $g$, 1-form $\alpha$ such that
\begin{equation}\label{eq:trans-def}
\begin{split}
\tilde{\eta} & =\eta + d^c \phi +d\psi +\alpha, \\
\tilde{\Phi} & =\Phi -\xi\otimes\tilde{\eta}\circ\Phi,\\
\tilde{g}    & =\frac{1}{2}d\tilde{\eta}\circ(\mathbb{1}\otimes\tilde{\Phi}) +\tilde{\eta}\otimes\tilde{\eta},
\end{split}
\end{equation}
and the transversal K\"{a}hler form becomes $\tilde{\omega}^T =\omega^T +\frac{1}{2}dd^c \phi$.
\end{lem}
\begin{proof}
We give only a sketch.  See~\cite{BoyGal08} for details.  The 1-form $\gamma=\tilde{\eta}-\eta$ is basic, and since
$d\gamma\in\Gamma(\Lambda_b^{1,1})$ and $\gamma$ is real, $d^c d\gamma =0$.  And we have the Hodge decomposition
\begin{equation}
\gamma =d^c \phi + d\psi +\alpha,
\end{equation}
with respect to the transversal K\"{a}hler metric $g^T$, where $\alpha\in\mathcal{H}^1_{g^T}$ is harmonic.
But note that $\mathcal{H}_{\R,g^T}^1 =\mathcal{H}_{\R,g}^1$, where the latter is the space of real harmonic 1-forms on $(M,g)$.
This is because a $\beta\in\Gamma(\Lambda^1(M))$ satisfying $d\beta =0$ and $\mathcal{L}_\xi \beta=0$ must be basic.
\end{proof}
\begin{rmk}\label{rmk:gauge-trans}
It is easy to check that the parameter $\psi$ in (\ref{eq:trans-def}) changes the structure only by a gauge transformation along
the leaves.  That is, if $\psi\in C^\infty_b(M)$, then $\exp(\psi\xi)^*\eta =\eta +d\psi$.
\end{rmk}

\subsection{transversely extremal metrics}

Given a basic $\phi\in C^\infty_b(M,\C)$, we define $\partial_g^{\#} \phi$ to be the $(1,0)$ component of the gradient, that is
\begin{equation}
g(\partial_g^{\#} \phi,\cdot) =\ol{\partial}\phi.
\end{equation}
In order for $\partial_g^{\#} \phi$ to be transversely holomorphic we need in addition $\ol{\partial}\partial_g^{\#} \phi =0$.
This is equivalent to the fourth-order transversally elliptic equation
\begin{equation}
L_g^b \phi := (\ol{\partial}\partial_g^{\#})^*\ol{\partial}\partial_g^{\#}\phi.
\end{equation}
As in the K\"{a}hler case we have
\begin{equation}\label{eq:hol-oper}
L_g^b \phi =\frac{1}{4}\bigl(\Delta_b^2 +(\rho^T,dd^c \phi) +2(\partial s^T)\contr\partial_g^{\#}\phi\bigr).
\end{equation}
We define the space of \emph{holomorphy potentials} to be  $\mathcal{H}^b_g :=\ker L_g^b$.

We denote by $\mathfrak{M}(\xi,\ol{J})$ the metrics associated with Sasaki structures in $\mathcal{S}(\xi,\ol{J})$.  We
define the Calabi functional just as in (\ref{eq:Calabi-funct}) by
\begin{equation}\label{eq:Calabi-Sasak}
\begin{array}{rcl}
\mathfrak{M}(\xi,\ol{J}) & \overset{\mathcal{C}}{\longrightarrow} & \R \\
g & \mapsto & \int_M s_g^2 \, d\mu_g
\end{array}
\end{equation}
We seek critical points of $\mathcal{C}$.  Because $\mathcal{C}$ only depends on the deformation of the transversal K\"{a}hler metric
$\tilde{\omega}^T =\omega +\frac{1}{2}dd^c \phi$ and not the other parameters in Lemma~\ref{lem:trans-def} and
Proposition~\ref{prop:Sasaki-Ric}.~\ref{eq:submer-scal} these critical
points are the transversely extremal metrics.  The Euler-Lagrange equation for $\mathcal{C}$ was worked out in~\cite{BoyGalSim08}.
\begin{prop}[\cite{BoyGalSim08}]
The fist derivative of $\mathcal{C}$ at $g$ along the path $\omega^T_t =\omega^T +t\frac{1}{2}dd^c \phi$ is
\[ \frac{d}{dt}\mathcal{C}(g_t)|_{t=0} = -4\int_M s_g \, (L_g^b \phi)\, d\mu_g.\]
\end{prop}
\begin{defn}
A Sasaki metric $g\in\mathfrak{M}(\xi,\ol{J})$ is \emph{extremal} if it is a critical point of (\ref{eq:Calabi-Sasak}).
Equivalently, the basic vector field $\partial_g^{\#} s_g$ is transversely holomorphic.
\end{defn}

\subsection{automorphism groups}

We consider the relevant automorphism groups and Lie algebras associated to a Sasaki structure $(g,\eta,\xi,\Phi)$.

Consider first the strictly pseudo convex CR structure $(D,J)$.  We denote the group of CR automorphisms by $\CR(D,J)$ and its
Lie algebra by $\Cr(D,J)$.  A fundamental result of~\cite{Sch95} classifies strictly pseudoconvex CR manifolds
for which $\CR(D,J)$ acts nonproperly.  We only need the result for compact $M$.
\begin{thm}[\cite{Sch95}]
Let $(M,D,J)$ be a compact strictly pseudoconvex CR manifold.  If $\CR(D,J)$ is not compact, then $(M,D,J)$ is CR diffeomorphic
to $\mathbb{S}^{2m+1}$ with the standard CR structure, in which case $\CR(D,J)=\PSU(m+1,1)$.
\end{thm}

It is useful to have the following alternative characterization of Sasaki structures.
\begin{prop}\label{prop:CR-Sasaki}
Let $(M,D,J)$ be a strictly pseudoconvex manifold.  If $\xi\in\Cr(D,J)$ is everywhere transversal to $D$, then, after possibly changing sign to $-\xi$, there is a unique Sasaki structure $(g,\eta,\xi,\Phi)$ with Reeb vector field $\xi$.
\end{prop}
\begin{proof}
Let $\eta$ be the unique 1-form with $\ker\eta =D$ and $\eta(\xi)=1$.  After possibly changing signs on $\xi$ and $\eta$ we have
$d\eta|_D >0$.  Since $\xi$ preserves the distribution $D$,
$\mathcal{L}_{\xi}\eta =0$ and $\xi$ is the Reeb vector field of $\eta$.  Then one can define $\Phi$ by $\Phi|_D =J$ and
$\Phi(\xi)=0$, and one has $\mathcal{L}_{\xi}\Phi =0$.  This latter condition and the integrability of $(D,J)$ implies that
$(g,\eta,\xi,\Phi)$ with $g$ defined in (\ref{eq:metric}) is Sasaki.  See~\cite{BoyGal08} for details.
\end{proof}

We have the subgroup of the diffeomorphism group preserving the foliation $\mathscr{F}_\xi$
\[ \Fol(M,\mathscr{F}_\xi) =\{\phi\in\Diff(M) : \phi_* \mathscr{F}_\xi \subset\mathscr{F}_\xi \},\]
with Lie algebra
\[\fol(M,\mathscr{F}_\xi)=\{X\in\CX(M) : [X,\xi]\subset\Gamma(L_\xi)\}. \]
Note that $\Fol(M,\mathscr{F}_\xi)$ is infinite dimensional as every $X\in\Gamma(L_\xi)$ is in $\fol(M,\mathscr{F}_\xi)$.
Any $\phi\in\Fol(M,\mathscr{F}_\xi)$ induces a map of bundles $\phi_* :\nu(\mathscr{F}_\xi)\rightarrow\nu(\mathscr{F}_\xi)$.
The subgroup of transversely holomorphic automorphism of $\mathscr{F}_\xi$ can be characterized as those which induce an automorphism
of the complex bundle $(\nu(\mathscr{F}_\xi), \ol{J})$
\[\Fol(M,\mathscr{F}_\xi,\ol{J}):=\{\phi\in\Fol(M,\mathscr{F}_\xi) : \phi_* \circ\ol{J} =\ol{J}\circ\phi_* \}.\]
Note that this group is also infinite dimensional, since any section in $L_\xi$ has a 1-parameter group in
$\Fol(M,\mathscr{F}_\xi,\ol{J})$.
We will denote the projection of $X\in\CX(M)$ to a section of $\nu(\mathscr{F}_\xi)$ by 
$\ol{X}\in\Gamma(\nu(\mathscr{F}_\xi))$.

\subsubsection{Transversely holomorphic vector fields}
The Lie algebra $\fol(M,\mathscr{F}_\xi,\ol{J})$ of $\Fol(M,\mathscr{F}_\xi,\ol{J})$ will be called the space of
\emph{transversely holomorphic} vector fields.
A transversely holomorphic vector field can be characterized more succinctly.
\begin{prop}
Let $(g,\eta,\xi,\Phi)$ be any Sasaki structure with Reeb vector field $\xi$ and transversely holomorphic structure $\ol{J}$.
Thus $\Phi$ satisfies (\ref{eq:trans-CD}).  Then $X\in \fol(M,\mathscr{F}_\xi,\ol{J})$ if and only if
\begin{equation}\label{eq:trans-hol}
\overline{[X,\Phi Y]} =\ol{J}\overline{[X,Y]},
\end{equation}
for all $Y\in\CX(M)$.
\end{prop}
Note that (\ref{eq:trans-hol}) implies that $X\in\fol(M,\mathscr{F}_\xi)$; that is, it is automatically foliate.  Also, condition
(\ref{eq:trans-hol}) is equivalent to the $(1,0)$ vector field
\begin{equation}\label{eq:trans-holom}
\Xi=\frac{1}{2}(\ol{X} -\sqrt{-1}\ol{J}\ol{X}) \in\Gamma(\nu(\mathscr{F}_\xi))
\end{equation}
satisfying the transverse Cauchy-Riemann equations.

Since $\fol(M,\mathscr{F}_\xi,\ol{J})$ is infinite dimensional we define $\hol^T(\xi,\ol{J})$ to be the image of
\begin{equation}
\begin{array}{rcl}
\fol(M,\mathscr{F}_\xi,\ol{J}) & \overset{\pi}{\longrightarrow} & \Gamma(\nu(\mathscr{F}_\xi)) \\
X & \mapsto & \ol{X}
\end{array}
\end{equation}
which is a finite dimensional complex Lie algebra.  We will use $\hol^T(\xi,\ol{J})$ to denote both transversally
holomorphic $(1,0)$ vector fields as in (\ref{eq:trans-holom}), or transversally real holomorphic vector fields
depending on the context.

The subspace $\hol^T(\xi,\ol{J})_0 \subseteq\hol^T(\xi,\ol{J})$ of sections with a zero will turn out to be a Lie
subalgebra.  As remarked in the proof of Lemma~\ref{lem:trans-def}
\[H^{1,0}_{b}(M/\mathscr{F}_\xi) \oplus H^{0,1}_{b}(M/\mathscr{F}_\xi) =\mathcal{H}^1_{g^T} = \mathcal{H}^1_g,\]
where on the left we have the basic Dolbeault cohomology, and we see that $H_1(M,\Z)\subset (H^{1,0}_{b})^*$ is a lattice.
As in K\"{a}hler geometry, we have the \emph{Albanese} variety
\[\Alb(M,\xi,\ol{J})=(H_{b}^{1,0})^* /H_1(M,\Z) =H^0_b(M,\Omega^1_b)^*/H_1(M,\Z),\]
and associated map
\[ \mu: M\rightarrow\Alb(M,\xi,\ol{J}),\]
which is a transversely holomorphic map if one considers $\Alb(M,\xi,\ol{J})$ to have the trivial foliation with the points as leaves.
Explicitly, if $p_0 \in M$ is a fixed point and $\beta_1,\ldots\beta_k \in H^0_b(M,\Omega^1_b)$ are a basis then
\begin{equation}\label{eq:Alb-map}
 \mu(p) =\Bigl(\int_{p_0}^p \beta_1,\ldots, \int_{p_0}^p \beta_k \Bigr)
\end{equation}
for any path $\gamma$ from $p_0$ to $p$.

Arguing just as in~\cite[Thm. 1]{LebSim94} we see that the image of
\begin{equation}
\partial^{\#}_g :\mathcal{H}^b_g \rightarrow\hol^T(\xi,\ol{J})
\end{equation}
is precisely $\hol^T(\xi,\ol{J})_0$.  Furthermore, (\ref{eq:Alb-map}) induces a group homomorphism
$\mu:\Fol(M,\mathscr{F}_\xi,\ol{J})\rightarrow\Aut(Alb(M,\xi,\ol{J}))$, with $\hol^T(\xi,\ol{J})_0 \subseteq\hol^T(\xi,\ol{J})$ the
ideal given by the kernel.

Similar to the K\"{a}hler case we have the following.
\begin{lem}\label{lem:pot-kill}
If $X\in\hol^T(\xi,\ol{J})_0$, then $X=\partial^{\#}_g f$ for an imaginary function $f\in\sqrt{-1}C^\infty_b(M)$ if and only
if $\re X$ is Killing for $g^T$.  If this is so, then $V=\re X \in\Gamma(\nu(\mathscr{F}_\xi))$ lifts to a vector field
$\tilde{V}\in\aut(g,\eta,\xi,\Phi)$.  Conversely, if $\tilde{V}\in\aut(g,\eta,\xi,\Phi)$, then $V=\pi(\tilde{V}) =\re \partial^{\#}_g f$, for the imaginary valued function $f=\sqrt{-1}\eta(\tilde{V})$.
\end{lem}
\begin{proof}
Suppose $V=\re \partial^{\#}_g f$ with $f$ imaginary valued.  Then $\frac{\sqrt{-1}}{2} \ol{J}^*df=V\contr g^T$, so
$\frac{\sqrt{-1}}{2}df =V\contr\omega^T$, which implies that $\mathcal{L}_V \omega^T =0$ and $V$ is Killing.

Suppose $V$ is Killing and $\partial^{\#}_g f =V-\sqrt{-1}\ol{J}V$, where $f=u+\sqrt{-1}v$.  Then
$\frac{1}{2}(du-\ol{J}^*dv)=V\contr g^T$,
which implies $\frac{1}{2}(-\ol{J}^*du -dv) =V\contr\omega^T$.  Since $\mathcal{L}_V \omega^T =0$, we have
$dd^c u=0$ and $u$ must be constant.

If $V=\re\partial^{\#}_g f$ with $f$ imaginary valued, choose $\tilde{V}\in\CX(M)$ with $\pi(\tilde{V})=V$ and
$\eta(\tilde{V})=-\sqrt{-1}f$.  Then
\[\mathcal{L}_{\tilde{V}}\eta =d(\eta(\tilde{V})) +\tilde{V}\contr d\eta =-\sqrt{-1}df +\tilde{V}\contr d\eta =0.\]
\end{proof}

\subsubsection{Real holomorphy potentials}
It will be useful to have a description of real holomorphic transversal vector fields and potentials.
Given a real $X\in\hol^T(\xi,\ol{J})$, let $\beta =X^{\flat}$.  If $\nabla^{-} \beta$ denotes the $\ol{J}$-anti-invariant component of
$\nabla^T \beta$.  Then a basic $X\in\Gamma(\nu(\mathscr{F}_\xi))$ is transversally holomorphic if and only if
$\nabla^{-} X=0$.  It follows that $dd^c \beta =0$, and we have the Hodge decomposition
\begin{equation}\label{eq:real-decom}
\beta =\beta_h +du_{X} +d^c v_{X},
\end{equation}
where $\beta_h$ is harmonic, $u_{X}$ and $v_{X}$ are real functions, and $d^c v_{X}$ is coclosed.
We have $X\in\hol^T(\xi,\ol{J})_0$, when $\beta_h =0$ in which case $f_{X} =u_{X} +\sqrt{-1}v_{X}$ is
the holomorphy potential of $X^{1,0}$, i.e. $X^{1,0} =\partial^{\#}_g f_{X}$.  Note that $X$ is Killing for $g^T$ if and only if
$u_{X}$ is constant.

We define the real operator $\mathbb{L}^b_g$ by
\[ \mathbb{L}^b_g f =(\nabla^- d)^*(\nabla^- d)f.\]
Then a real basic function $f$ satisfies $\mathbb{L}^b_g f =0$ if and only if $\grad f$ is real holomorphic.  And every
$X\in\hol^T(\xi,\ol{J})_0$ Killing with respect to $g^T$ can be written $X=\ol{J}\grad f$ for such an $f$.
As shown in~\cite{RolSimTip11} we have
\begin{equation}
\mathbb{L}^b_g f =\frac{1}{2}\Delta_b^2 f +\frac{1}{2}(\rho^T, dd^c f)+ \frac{1}{2}(df,ds_g),
\end{equation}
and comparing with (\ref{eq:hol-oper}) is related to $L_g^b$ by
\begin{equation}
2L_g^b f =\mathbb{L}^b_g f +\frac{\sqrt{-1}}{2}\mathcal{L}_{\ol{J}\grad s_g} f.
\end{equation}
\begin{lem}[\cite{RolSimTip11}]
The space of real basic solutions of $\mathbb{L}^b_g$ coincides with the space of real basic solutions of $L_g^b$.
\end{lem}

As in Lemma~\ref{lem:pot-kill} we have the correspondence
\begin{equation}\label{eq:pot-Kill}
\begin{array}{rcl}
\mathcal{H}_g^b \cap C^\infty_b(M,\R) & \simarrow & \aut(g,\eta,\xi,\Phi) \\
v & \longmapsto & X \text{ s.t. }\eta(X)=v,\text{ and }\pi(X)=\frac{1}{2}\ol{J}\grad v.
\end{array}
\end{equation}

\subsubsection{Automorphisms of Sasaki-extremal manifolds}
We recall the structure of \\
$\hol^T(\xi,\ol{J})$ when $(g,\eta,\xi,\Phi)$ is Sasaki-extremal~\cite{BoyGalSim08} as it will important
to what follows.  The result is similar to the theorem of Calabi~\cite{Cal85} on the automorphism group of a K\"{a}hler manifold with
an extremal metric.
\begin{thm}
Let $(g,\eta,\xi,\Phi)\in\mathcal{S}(\xi,\ol{J})$ be a Sasaki-extremal structure.
Then we have the semidirect sum decomposition
\begin{equation}
\hol^T(\xi,\ol{J}) =\mathfrak{a}\oplus\hol^T(\xi,\ol{J})_0,
\end{equation}
where $\mathfrak{a}$ is the Lie algebra of parallel, with respect to $g^T$, sections of $\nu(\mathscr{F}_\xi)$.  And we also have
\begin{equation}\label{eq:extr-decomp}
\hol^T(\xi,\ol{J})_0 =\mathfrak{k}\oplus\ol{J}\mathfrak{k}\oplus\bigl(\bigoplus_{\lambda>0}\mathfrak{h}^{\lambda} \bigr),
\end{equation}
where $\mathfrak{k}=\aut(g,\eta,\xi,\Phi)/{\xi}$ is the image under $\partial_g^{\#}$ of the imaginary valued functions in $\mathcal{H}^b_g$
and $\mathfrak{h}^\lambda =\{\ol{X}\in\hol^T(\xi,\ol{J})_0 : [\partial^{\#}_g s_g ,\ol{X}]=\lambda\ol{X}\}$
and $\mathfrak{k}\oplus\ol{J}\mathfrak{k} =C_{\hol^T(\xi,\ol{J})_0}(\partial^{\#}_g s_g)$, the centralizer of $\partial^{\#}_g s_g$.

Furthermore, the connected component of the identity $G=\Aut(g,\eta,\xi,\Phi)_0 \subset\Fol(M,\mathscr{F}_\xi,\ol{J})$ is a maximal
compact connected subgroup.  And any other maximal compact connected subgroup is conjugate to $G$ in $\Fol(M,\mathscr{F}_\xi,\ol{J})$.
\end{thm}
\begin{proof}
Everything but the final statement is proved in~\cite{BoyGalSim08}.  The last statement was proved in~\cite{Cal85} in the K\"{a}hler
case, and easily follows from the theory of finite dimensional Lie groups.  It is not as simple in this case as
$\Fol(M,\mathscr{F}_\xi,\ol{J})$ is infinite dimensional, and furthermore is not even known to have a Fr\'echet Lie group structure.

Let $G'\subset\Fol(M,\mathscr{F}_\xi,\ol{J})$ be any maximal connected compact subgroup with Lie algebra $\mathfrak{g}'$.
By applying the familiar averaging argument using a
Haar measure on $G'$ to $(g,\eta,\xi,\Phi)$ we get a Sasaki structure $(\tilde{g},\tilde{\eta},\tilde{\xi},\tilde{\Phi})$ with
$G' =\Aut(\tilde{g},\tilde{\eta},\tilde{\xi},\tilde{\Phi})_0$.  

It is proved in~\cite{Noz11} that there exist an $f\in\Fol(M,\mathscr{F}_\xi,\ol{J})$ so that $f^*\tilde{\eta} (\xi) =c\in\R_{>0}$.
This follows from a leaf wise version of Moser's argument which can be used to prove $f^*\tilde{\eta}|_{L_\xi} =c\eta|_{L_\xi}$.
But since the averaging preserves the volume
\begin{equation}
\begin{split}
\Vol(M,\tilde{g}) & = \frac{1}{m!}\int_M f^*\tilde{\eta}\wedge(\frac{1}{2}f^* d\tilde{\eta})^m \\
                  & = \frac{1}{m!}\int_M c^{m+1} \eta\wedge(\frac{1}{2}d\eta)^m \\
                  & =c^{m+1}\Vol(M,g),
\end{split}
\end{equation}
so $c=1$.  The second equality follows because $f^*\tilde{\eta}-c\eta$ is a basic form and an application of Stokes theorem.  
Therefore $f$ applied to $(\tilde{g},\tilde{\eta},\tilde{\xi},\tilde{\Phi})$
gives a Sasaki structure $(g',\eta',\xi,\Phi')\in\mathcal{S}(\xi,\ol{J})$ with $\Aut(g',\eta',\xi,\Phi')_0$ conjugate
to $G'$ in $\Fol(M,\mathscr{F}_\xi,\ol{J})$.

We have the continuous group homomorphism
\[\Upsilon:\Fol(M,\mathscr{F}_\xi,\ol{J})\rightarrow\Aut(\hol^T(\xi,\ol{J})_0),\]
with $\Upsilon(\phi)\ol{X}=\phi_* \ol{X}$, where $\Fol(M,\mathscr{F}_\xi,\ol{J})$ is given the topology as a closed subgroup of the
diffeomorphism group.  By considering 1-parameter subgroups generated by $X$ such that $\ol{X}\in\hol^T(\xi,\ol{J})_0$, one sees that
the adjoint group of Lie algebra $\hol^T(\xi,\ol{J})_0,\
H=\Inn(\hol^T(\xi,\ol{J})_0) \subseteq\Aut(\hol^T(\xi,\ol{J})_0),$ is in the image of $\Upsilon$.
Note that the Lie algebra $\mathfrak{h}$ of $H=\Inn(\hol^T(\xi,\ol{J})_0)$ is $\hol^T(\xi,\ol{J})_0/Z(\hol^T(\xi,\ol{J})_0)$.

Let $G=\Aut(g,\eta,\xi,\Phi)_0$ for a Sasaki-extremal structure, then by (\ref{eq:extr-decomp}) it is easy
to see $G$ is maximal connected compact.  Let $G'\subset\Fol(M,\mathscr{F}_\xi,\ol{J})$ 
be any other connected maximal compact subgroup with Lie algebra $\mathfrak{g}'$.
As shown above we may assume, up to conjugation, that $G'=\Aut(g',\eta',\xi,\Phi')_0$, for
$(g',\eta',\xi,\Phi')\in\mathcal{S}(\xi,\ol{J})$.
Let $\ol{\mathfrak{g}}':=\pi(\mathfrak{g}')=\mathfrak{g}'/{\xi} \subset\hol^T(\xi,\ol{J})_0$.
Then $\mathfrak{g}'$ has image under $\Upsilon$ given by 
$\Upsilon_* \mathfrak{g}' =\ol{\mathfrak{g}}'/{\ol{\mathfrak{g}}'\cap Z(\hol^T(\xi,\ol{J})_0)}$.
Since $\Upsilon_* \mathfrak{g}\subseteq\mathfrak{h}$ is the Lie algebra of a maximal compact subgroup of $H$, there exists an 
$h\in H$ so that $\Ad(h)_* \Upsilon_* \mathfrak{g}' \subseteq\Upsilon_* \mathfrak{g}$. 

Let $\phi\in\Fol(M,\mathscr{F}_\xi,\ol{J})$ be such that $\Upsilon(\phi)=h$.  Then $\hat{\mathfrak{g}}=\Ad(\phi)_*(\mathfrak{g}')$ satisfies $\Upsilon_* \hat{\mathfrak{g}}\subseteq\Upsilon_* \mathfrak{g}$.  Since $\pi(\mathfrak{g})$ contains $Z(\hol^T(\xi,\ol{J})_0)$,
we have $\pi(\hat{\mathfrak{g}})\subseteq\pi(\mathfrak{g})$.  Applying $\phi$ to $(g',\eta',\xi,\Phi')$ and then a transformation
as $f$ above, we get a Sasaki structure $(\tilde{g},\tilde{\eta},\xi,\tilde{\Phi})$ with 
$\Aut(\tilde{g},\tilde{\eta},\xi,\tilde{\Phi})_0 =\tilde{G}$ conjugate to $G'$ in $\Fol(M,\mathscr{F}_\xi,\ol{J})$ and 
with $\pi(\tilde{\mathfrak{g}})\subseteq\pi(\mathfrak{g})$.

From Lemma~\ref{lem:trans-def} we have $\eta =\tilde{\eta} +d\psi +d^c \phi +\alpha$, with $\psi,\ \phi$ basic and $\alpha$
$\tilde{g}$ harmonic.  So $\omega^T =\tilde{\omega}^T +\frac{1}{2}dd^c\phi$.  
After a gauge transformation (See Remark~\ref{rmk:gauge-trans}) we may assume that $\psi=0$.
Any $\ol{X}\in\pi(\tilde{\mathfrak{g}})\subseteq\pi(\mathfrak{g})$ is both $g^T$ and $\tilde{g}^T$ Killing.  Since
\[0=\mathcal{L}_{\ol{X}}dd^c\phi =dd^c \ol{X}\phi,\]
$\ol{X}\phi$ is constant, and thus $\ol{X}\phi=0$.  If $X\in\tilde{\mathfrak{g}}$, then $d(\tilde{\eta}(X))=-X\contr d\tilde{\eta}$.  
We have
\begin{equation}
\begin{split}
d(\eta(X)) & =d(\tilde{\eta}(X) +d^c \phi(X) +\alpha(X)) \\
                 & =-X\contr d\tilde{\eta} +d(d^c \phi(X)) \\
                 & =-X\contr d\tilde{\eta} -(X\contr dd^c\phi )\\
                 & =-X\contr d\eta,
\end{split}
\end{equation}
where the second equality is because $\alpha$ is harmonic and the third because $\mathcal{L}_X d^c \phi =0$.
Therefore $\tilde{\mathfrak{g}}\subseteq\mathfrak{g}$.  Since $\tilde{G}$ is maximal, we must have $\tilde{G}=G$.
Therefore $G'$ is conjugate to $G$ in $\Fol(M,\mathscr{F}_\xi,\ol{J})$.
\end{proof}
\begin{rmk}
If we have two Sasaki-extremal structures $(g_i,\eta_i,\xi,\Phi_i)\in\mathcal{S}(\xi,\ol{J}),\ i=1,2$, then there is a
$\phi\in\Fol(M,\mathscr{F}_\xi,\ol{J})$ so that $\phi_*(g_2,\eta_2,\xi,\Phi_2)=(\hat{g},\hat{\eta},\xi,\hat{\Phi})$ satisfies $\Aut(\hat{g},\hat{\eta},\xi,\hat{\Phi})_0 =\Aut(g_1,\eta_1,\xi,\Phi_1)_0$.  One should be able to extend the proof of uniqueness
of extremal K\"{a}hler metrics~\cite{CheTia08} to the Sasaki case.
\end{rmk}

From now on $\G'\subset\Fol(M,\mathscr{F}_\xi,\ol{J})$ will be a fixed maximal connected compact subgroup,
and $G\subseteq G'$ a connected compact subgroup with Lie algebras $\{\xi\}\subseteq\mathfrak{g}\subseteq\mathfrak{g}'$.
As seen above there is a Sasaki structure $(g,\eta,\xi,\Phi)$ with $G'=\Aut(g,\eta,\xi,\Phi)_0$.  We define several Lie algebras:
\begin{itemize}
\item  $\mathfrak{z}=Z(\mathfrak{g})$, the center of $\mathfrak{g}$,
\item  $\mathfrak{z}' =C_{\mathfrak{g}'}(\mathfrak{g})$, the centralizer of $\mathfrak{g}$ in $\mathfrak{g}'$,
\item  $\mathfrak{z}''=C_{\hol^T(\xi,\ol{J})_0}(\mathfrak{g})$, the centralizer of $\mathfrak{g}$ in $\hol^T(\xi,\ol{J})_0$,
\item  $\mathfrak{p} =N_{\mathfrak{g}'}(\mathfrak{g})$, the normalizer of $\mathfrak{g}$ in $\mathfrak{g}'$,
\item  $\mathfrak{q} =N_{\hol^T(\xi,\ol{J})_0}(\mathfrak{g})$, the normalizer of $\mathfrak{g}$ in $\hol^T(\xi,\ol{J})_0$.
\end{itemize}

We denote by $\mathcal{H}_g^{\mathfrak{s}} \subseteq\mathcal{H}_g^b$ the corresponding space of holomorphy potentials where
$\mathfrak{s}$ is one of the above Lie algebras.  Note that
\[\mathcal{H}_g^{\mathfrak{z}}\subseteq\mathcal{H}_g^{\mathfrak{z}'} \subseteq\mathcal{H}_g^{\mathfrak{p}} \subseteq\mathcal{H}_g^{\mathfrak{g}'}\]
consist of purely imaginary functions, and $\mathcal{H}_g^{\mathfrak{z}}$ (respectively $\mathcal{H}_g^{\mathfrak{z}'}$) consist of $G$-invariant functions in $\mathcal{H}_g^{\mathfrak{g}}$ (respectively $\mathcal{H}_g^{\mathfrak{g}'}$).

We also have the following whose proof is just as in~\cite{RolSimTip11}.
\begin{lem}
We have the isomorphisms of Lie algebras induced by the injections
\[ \mathfrak{z}'/\mathfrak{z} \cong\mathfrak{p}/\mathfrak{g} ,\quad \mathfrak{z}''/\mathfrak{z} \cong\mathfrak{q}/\mathfrak{g} .\]
\end{lem}

\subsection{Relative Futaki invariant}

\subsubsection{Reduced scalar curvature}
We let $L^2_k(M)$ denote the k\superscript{th} real Sobolev space.  We assume $k>m+1$, where $\dim M=2m+1$, so that
$L^2_k(M)$ is a Banach algebra.  Denote $L^2_{k,G}(M)$ to be the subspace of $G$-invariant functions in $L^2_k(M)$, which
is also a Hilbert space and Banach algebra.  Note that $L^2_{k,G}(M)$ are basic functions since we assume that $\xi\subset\mathfrak{g}$.
The $L^2$-inner product is defined using the metric of the $G$-invariant Sasaki structure.

We have an orthogonal decomposition
\[ L^2_{k,G}(M)=\sqrt{-1}\mathcal{H}_g^{\mathfrak{z}} \oplus W_{g,k}, \]
with the projections
\[\pi_g^G :L^2_{k,G}(M)\rightarrow\sqrt{-1}\mathcal{H}_g^{\mathfrak{z}}\text{  and  }\pi_g^W :L^2_{k,G}(M)\rightarrow W_{g,k}.\]
Associated to a $G$-invariant Sasaki structure $(g,\eta,\xi,\Phi)$ we define the \emph{reduced scalar curvature}
\begin{equation}
s_g^G =\pi_g^W (s_g).
\end{equation}
The condition $s_g^G =0$ is equivalent to $s_g \in\sqrt{-1}\mathcal{H}_g^{\mathfrak{z}}\subset\mathcal{H}_g^b$, so this implies
that $g$ is Sasaki-extremal.

\subsubsection{Reduced Ricci form and Ricci potential}
As in~\cite{Sim00,Sim05} we define the \emph{reduced Ricci form} and Ricci potential.  Let $L^2_{k,G}(\Lambda^{1,1}_b M)$ be the space
of basic $G$-invariant $(1,1)$-forms in $L^2_k$.  As in~\cite{Sim05} one can define a projection
\[ \Pi_g^G :L^2_{k,G}(\Lambda^{1,1}_b M) \rightarrow L^2_{k,G}(\Lambda^{1,1}_b M), \]
by
\[  \Pi_g^G \gamma =\gamma +dd^c f ,\]
where $f =G_g(\pi^W(\omega^T ,\gamma))$.  This projection intertwines the trace with $\pi_g^G$, that is
$(\omega^T,\Pi_g^G \gamma) =\pi^G_g (\omega^T,\gamma)$.

We obtain the \emph{reduced Ricci form}~\cite{Sim05} by
\begin{equation}\label{eq:Ricci-form}
\rho^G =\Pi_g^G \rho^T,
\end{equation}
and the related identity
\begin{equation}\label{eq:Ricci-pot}
\rho^G =\rho^T +\frac{1}{2}dd^c \psi^G_g.
\end{equation}
One has $\psi^G_g =2G_g(\pi^W(\omega^T ,\rho^T))$ and $\rho^G =\rho^T$ if and only if $s_g^G =0$.

\subsubsection{Relative Futaki invariant}
Suppose we have a $G$-invariant Sasaki structure $(g,\eta,\xi,\Phi)$ on $M$.  We define the \emph{relative Futaki
invariant}
\begin{equation}\label{eq:rel-Futaki}
\mathcal{F}_{G,\xi}(X) =\int_M d^c \psi^G_g(X)\, d\mu_g,
\end{equation}
where $X\in\hol^T(\xi,\ol{J})$ is any real transversely holomorphic vector field and $\psi^G_g$ is the Ricci potential
(\ref{eq:Ricci-pot}) of $g$.  Though defined in terms of the metric, (\ref{eq:rel-Futaki}) is independent of the
$G$-invariant Sasaki structure in $\mathcal{S}(\xi,\ol{J})$.  See~\cite{RolSimTip11}, and also~\cite{BoyGalSim08} and~\cite{FutOnWan09}
for the Sasaki-Futaki invariant,
where $\psi^G_g$ is replaced by the usual Ricci-potential $\psi_g =2G_g((\omega^T ,\rho^T)-(\omega^T ,\rho^T)_0)$, with
\[\begin{split}
(\omega^T ,\rho^T)_0 & =\frac{\int_M (\omega^T ,\rho^T)\, d\mu_g}{\int_M d\mu_g} \\
					&  =\frac{\int_M  s^T_g\, d\mu_g}{\int_M d\mu_g}  \\
					& = \frac{4m\pi c_1(\mathscr{F}_\xi)\cup[\omega^T]^{m-1}}{[\omega^T]^m}.
\end{split}\]
the average of the scalar curvature.

In terms of the Hodge decomposition (\ref{eq:real-decom}) of the dual 1-form $X^{\flat} =X^\flat_h +du_X +d^c v_X$ we have
\begin{equation}\label{eq:rel-Futaki-form}
\begin{split}
\mathcal{F}_{G,\xi}(X) & =\int_M (J^*X^\flat_h+J^*du_X +J^*d^c v_X, d\psi^G_g)\, d\mu_g \\
					   & =\int_M (J^*X^\flat_h -d^c u_X +dv_X, d\psi^G_g)\, d\mu_g \\
				       & =\int_M (dv_X, d\psi^G_g)\, d\mu_g \\					
					   & =\int_M (v_X, \Delta_g \psi^G_g)\, d\mu_g \\	
					   & =\int_M v_X s_g^G\, d\mu_g. 	
\end{split}
\end{equation}
The third equality follows because $J^*X^\flat_h$ is harmonic and $d^c u_X$ is coclosed.

It follows from (\ref{eq:rel-Futaki-form}) that if $X\in\mathfrak{g}$ then $\mathcal{F}_{G,\xi}(X)=0$.  Thus we have
the $\R$-linear character
\begin{equation}\label{eq:rel-Futaki-map}
\mathcal{F}_{G,\xi}:\mathfrak{q}/\mathfrak{g} \rightarrow\R.
\end{equation}

A $G$-invariant Sasaki-extremal structure in $\mathcal{S}(\xi,\ol{J})$ has $s_g^G =0$ if and only if (\ref{eq:rel-Futaki-map})
vanishes.

\section{Deformations of Sasaki structures}

Besides the transversal K\"{a}hler deformations of a Sasaki structure $(g,\eta,\xi,\Phi)$ considered in Lemma~\ref{lem:trans-def}
we will consider two other deformations.  First, we will consider deformations of the transversal complex structure $\ol{J}$ on
$\mathscr{F}_\xi$.  In particular, we will consider deformations equivariant with respect to the compact group $G$.
Second, we will also consider deformations of the Reeb vector field $\xi$.  Together these give a subspace of the 
versal deformation space of $(\mathscr{F}_\xi,\ol{J})$ as a transversely holomorphic foliation.  

\subsection{Deformation of foliations}\label{subsect:defor}

\subsubsection{Kuranishi space}\label{subsubsec:Kurani}
We consider the deformations of the transversely holomorphic foliation $(\mathscr{F}_\xi,\ol{J})$.  In particular, we are interested
in the deformations of $(\mathscr{F}_\xi,\ol{J})$ that preserve its structure as a smooth foliation.  The existence of
a versal space for deformations, which fix the smooth foliation structure, was proved in~\cite{KacNic89}, and the universal
property of the versal space was strengthened in~\cite{Gir92}.  Note this requires an assumption on the foliation, of which
being transversally Hermitian is sufficient, which is clearly the case for $(\mathscr{F}_\xi,\ol{J})$.

We denote by $\mathcal{A}^{0,k} =\Gamma(\Lambda^{0,k}_b\otimes\nu(\mathscr{F})^{1,0})$ the space of smooth basic forms of
type $(0,k)$ with values in $\nu(\mathscr{F})^{1,0}$, and we have the Dolbeault complex
\begin{equation}\label{eq:Dol-comp}
 0\rightarrow \mathcal{A}^{0,0} \overset{\ol{\partial}_b}{\longrightarrow}\mathcal{A}^{0,1} \overset{\ol{\partial}_b}{\longrightarrow}\cdots.
\end{equation}
The tangent space to the versal space is the first cohomology of $(\mathcal{A}^{0,\bullet},\ol{\partial}_b)$ denoted
$H^1_{\ol{\partial}_b}(\mathcal{A}^{0,\bullet})$.

The versal space $\mathcal{V}$ is the germ of $\theta^{-1}(0)$ where $\theta$ is an analytic map
\begin{equation}
 H^1_{\ol{\partial}_b}(\mathcal{A}^{0,\bullet}) \overset{\theta}{\rightarrow} H^2_{\ol{\partial}_b}(\mathcal{A}^{0,\bullet}).
\end{equation}
Thus there exists a family of transverse holomorphic structures on $\mathscr{F}_\xi$ parametrized by $\mathcal{V}$,
$(\mathscr{F}_\xi ,\ol{J}_t)_{t\in\mathcal{V}}$, such that any other deformation is given by a pull-back via a map to
$\mathcal{V}$.

As above, we consider a compact group $G$ acting on $(\mathscr{F}_\xi,\ol{J})$.  One can consider the complex of
$G$-invariant forms $\mathcal{A}^{0,k}_G =\Gamma(\Lambda^{0,k}_b\otimes\nu(\mathscr{F})^{1,0})^G$,
\begin{equation}\label{eq:Dol-comp-inv}
 0\rightarrow \mathcal{A}_G^{0,0} \overset{\ol{\partial}_b}{\longrightarrow}\mathcal{A}_G^{0,1} \overset{\ol{\partial}_b}{\longrightarrow}\cdots.
\end{equation}
By considering Hodge theory for transversally elliptic operators one can show that the cohomology of (\ref{eq:Dol-comp-inv}) is
naturally identified with $H^k_{\ol{\partial}_b}(\mathcal{A}^{0,\bullet})^G$, the cohomology classes fixed by $G$ of (\ref{eq:Dol-comp}).
The tangent space of the subspace $\mathcal{V}^G \subseteq\mathcal{V}$ of $G$-equivariant deformations is
$H^1_{\ol{\partial}_b}(\mathcal{A}^{0,\bullet})^G$.

\subsubsection{Sasaki structures}

Suppose we have a family of $G$-invariant transversal complex structures
$(\mathscr{F}_\xi,\ol{J}_t),\ t\in\mathcal{B}\subseteq\mathcal{V}^G$, with $\mathcal{B}$ a smooth subspace.
By the results of~\cite{KacGmi97}, which extend the stability result of Kodaira and Spencer on deformations of K\"{a}hler manifolds
to deformations of foliations fixing the differentiable structure, we have transverse K\"{a}hler structures
$\omega^T_t$ on $(\mathscr{F}_\xi,\ol{J}_t)$ with $t,\in\mathcal{B}$, after possibly shrinking $\mathcal{B}$.
Although the transversely K\"{a}hler property is stable, there is a further obstruction to the existence of a Sasaki structure
compatible with $(\mathscr{F}_\xi,\ol{J}_t)$ for $t\in\mathcal{B}$.  The necessary and sufficient condition for the 
existence of a Sasaki structure were obtained in~\cite{Noz11} for more general deformations, not necessarily preserving
the smooth structure of $\mathscr{F}$.  But for our purposes, we will only consider deformations of $(\mathscr{F}_\xi,\ol{J})$
preserving the smooth foliation structure.

\begin{defn}\label{defn:Euler-obst}
A deformation $(\mathscr{F}_\xi,\ol{J}_t),\ t\in\mathcal{B}\subseteq\mathcal{V}$ of the underlying foliation of a Sasaki
structure $(g,\eta,\xi,\Phi)$ is of \emph{$(1,1)$-type} if
for all $t\in\mathcal{B}$ the $(0,2)$-component of the Euler class $[d\eta^{0,2}]\in H^{0,2}_{b,t}(M/\mathscr{F}_\xi)$ is zero, where
$H^{0,2}_{b,t}(M/\mathscr{F}_\xi)=H^2_{\ol{\partial}_{b,t}}(\Gamma(\Lambda_b^{0,\bullet}))$ is the basic Dolbeault cohomology for the
transversal complex structure $\ol{J}_t$.
\end{defn}
 
\begin{thm}[\cite{Noz11}]\label{thm:Euler-obst}
Let $(\mathscr{F}_\xi,\ol{J}_t),\ t\in\mathcal{B}\subseteq\mathcal{V}$ be a deformation of the Reeb foliation of
$(g,\eta,\xi,\Phi)$.  Then there exists a smooth
family of $(g_t,\eta_t,\xi,\Phi_t)\in\mathcal{S}(\xi,\ol{J}_t),\ t\in V\subset\mathcal{B}$ of compatible Sasaki structures,
where $V$ is a neighborhood of zero in $\mathcal{B}$, if and only if the deformation is of $(1,1)$-type restricted to $V$.  
\end{thm}

An application of a transversal Kodaira-Nakano vanishing theorem gives the following which is basically Corollary 1.4
of~\cite{Noz11}.
\begin{prop}\label{prop:obst-pos}
Let $(\mathscr{F}_\xi,\ol{J}_t),\ t\in\mathcal{B}$ be a deformation of the underlying foliation of a Sasaki
structure $(g,\eta,\xi,\Phi)$, and suppose the first Chern class 
$c_1^b(\mathscr{F}_\xi)=\frac{1}{2\pi}[\rho^T]\in H^2_b(M/\mathscr{F}_\xi ,\R)$ is representable by a basic positive
$(1,1)$-form, then after restricting to a neighborhood of zero $V\subset\mathcal{B}$ the deformation is of $(1,1)$-type.
\end{prop} 
\begin{rmk}
In particular, the proposition is applicable if $(g,\eta,\xi,\Phi)$ is Sasaki-Einstein or more generally $\Ric^T >0$.
But in order to apply Kodaira-Nakano vanishing to prove $H^{0,2}_{b,t}(M/\mathscr{F}_\xi)$ for $t\in V$ we only
need that the transverse anti-canonical bundle $\bigwedge ^{m,0} \nu(\mathscr{F}_\xi)$ is positive.
\end{rmk}

\begin{xpl}
Let $Z_a=\C^2/\Lambda$ be the complex torus given by the lattice $\Lambda =\Z\{\lambda_1,\ldots,\lambda_4\}\subset\C^2$ with
$\lambda_1 =(1,0),\lambda_2 =(0,1),\lambda_3 =(i,0),\lambda_4 =(a,i)$ where $a\in\C$.  Let $x_1,\ldots, x_4$ be dual real
coordinates to the lattice vectors $\lambda_1,\ldots,\lambda_4$.  Then 
\[ \omega =dx_1 \wedge dx_3 + dx_2 \wedge dx_4 \]
is integral, $[\omega]\in H^2(Z,\Z)$, and so defines a smooth $S^1$ bundle $\mathbf{L}$ with total space $M$.  
Let $z_1,z_2$ be the standard holomorphic coordinates on $\C^2$.  Then a routine calculation gives
\begin{equation*}
\omega =\frac{i}{2} dz_1 \wedge d\ol{z}_1 +\frac{i}{2}dz_2 \wedge d\ol{z}_2 +\frac{a}{4} d\ol{z}_1 \wedge dz_2
+\frac{\ol{a}}{4} dz_1 \wedge d\ol{z}_2 -\frac{\ol{a}}{4} dz_1 \wedge dz_2 -\frac{a}{4} d\ol{z}_1 \wedge d\ol{z}_2
\end{equation*}
and $[\omega^{0,2}]=-\frac{a}{4}[d\ol{z}_1 \wedge d\ol{z}_2]$ is nonzero in $H^{0,2}(Z)$ for $a\neq 0$.
When $a=0$ as a $\C$-bundle $\mathbf{L}$ has a natural holomorphic structure and polarizes $Z_0$, and $M$ has a natural
Sasaki structure with transversal K\"{a}hler form $\omega$ and Reeb foliation $\mathscr{F}_\xi$ given by the 
$S^1$ bundle $\mathbf{L}$ with leaf space $Z_0$. 
But for $a\neq 0$ there is no complex structure on the $\C$-bundle $\mathbf{L}$ and no compatible Sasaki structure on $M$ with
$\mathscr{F}_\xi$ given by $\mathbf{L}$ with leaf space $Z_a$.

In fact, one can prove that for $a\in\C\setminus\Q +i\Q$ there is no integral nondegenerate $(1,1)$-form on $Z_a$.
Thus $Z_a$ is not algebraic for $a\in\C\setminus\Q +i\Q$.
\end{xpl}

If we have a $G$-equivariant deformation of $(1,1)$-type $(\mathscr{F}_\xi,\ol{J}_t),\ t\in\mathcal{B}\subseteq\mathcal{V}^G$,
then the family $(g_t,\eta_t,\xi,\Phi_t),\ t\in\mathcal{B}$ of Theorem~\ref{thm:Euler-obst} can be taken to be $G$-invariant
by averaging $\eta_t$ by the $G$-action.  In the following we will assume the deformed structures $(g_t,\eta_t,\xi,\Phi_t)$
are $G$-invariant.

\pagebreak

\subsection{Sasaki cone}

\subsubsection{Deforming the Reeb vector field in the Sasaki cone}\label{subsubsec:Sasak-cone}
We have a $G$-invariant Sasaki structure $(g,\eta,\xi,\Phi)$, where $G$ has Lie algebra $\mathfrak{g}$ with center
$\mathfrak{z}$ with $\xi\in\mathfrak{z}$.
\begin{defn}
We define the \emph{Sasaki cone} of $\mathfrak{z}\subseteq\aut(g,\eta,\xi,\Phi)$ to be
$\mathfrak{z}^+ =\{\zeta\in\mathfrak{z} :\ \eta(\zeta)>0\}$, which is clearly open in $\mathfrak{z}$.
\end{defn}

If $\zeta\in\mathfrak{z}^+$, then $\eta_\zeta =\eta(\zeta)^{-1}\eta$ is a contact form
for $D =\ker\eta$ with Reeb vector field $\zeta$.  It follows from Proposition~\ref{prop:CR-Sasaki} that
$(g_\zeta, \eta_\zeta, \zeta, \Phi_\zeta)$ is a Sasaki structure with the same underlying CR structure,
where $\Phi_\zeta (X) =\Phi(X)-\eta_\zeta(X)\Phi(\zeta)$ and $g_\zeta$ is defined in (\ref{eq:metric}).

Let $T^r \subseteq G$ be the connected component of the identity of the center.
We get an alternative description of $\mathfrak{z}^+$ if we consider the moment map for the Hamiltonian action on the cone
$C(M)=\R_{>0}\times M$.  In fact the moment map for the symplectic action of $T^r$ on $C(M)$ is given in terms of the contact form by
\begin{equation} \label{eq:moment-map2}
 \mu_{\eta} :C(M) \rightarrow\mathfrak{z}^*,
\end{equation}
where $\mu_{\eta}(x,r)(X) = r^2\eta_x(X)$ with $X\in\mathfrak{z}$ also denoting the vector field induced on $C(M)$.
The image of (\ref{eq:moment-map2}) is a strongly convex rational polyhedral cone $\mathcal{C}_{\mathfrak{z}}^*\subset\mathfrak{z}^*$ (\cite{MorTom97}).
Although the map $\mu_\eta$ depends on the contact form $\eta$, the image $\mathcal{C}_{\mathfrak{z}}^*$ is independent of
transversal K\"{a}hler deformations, considered in Lemma~\ref{lem:trans-def}, and the deformations of the Reeb vector field
$\xi\in\mathfrak{z}^+$ considered above.

By Farkas' theorem, the dual cone $\mathcal{C}_{\mathfrak{z}}$ to $\mathcal{C}_{\mathfrak{z}}^*$ is also a strongly convex polyhedral cone.  From the definition of $\mu_\eta$ we see that
\begin{equation}
\mathfrak{z}^+ =\overset{\circ}{\mathcal{C}_{\mathfrak{z}}}.
\end{equation}

We will consider deformations of the Sasaki structures $(g_t,\eta_t,\xi,\Phi_t)\in\mathcal{S}(\xi,\ol{J}_t),\ t\in\mathcal{B}$
of the previous section.  Given $\phi\in L^2_{k,G}(M)$, with $k>m+5$, and $\xi_\alpha =\xi+\alpha\in\mathfrak{z}^+$ we consider the Sasaki structure
$(g_{t,\alpha,\phi},\eta_{t,\alpha,\phi},\xi_\alpha,\Phi_{t,\alpha,\phi})\in\mathcal{S}(\xi_\alpha,\ol{J}_{t})$ with
\begin{align}
\eta_{t,\alpha,\phi} & =\eta_{t,\xi_\alpha} +d^c \phi, \label{eq:Sasaki-var1}\\
\Phi_{t,\alpha,\phi} & =\Phi_{t,\xi_\alpha} -(\xi_\alpha\otimes(\eta_{t,\xi_\alpha,\phi} -\eta_{t,\xi_\alpha})\circ\Phi_{t,\xi_\alpha},\label{eq:Sasaki-var2}
\end{align}
and $g_{t,\alpha,\phi}$ defined from (\ref{eq:Sasaki-var1}) and (\ref{eq:Sasaki-var2}) as in (\ref{eq:metric}).
Therefore we have a space of Sasaki structures parametrized by $(t,\alpha,\phi)\in\mathcal{B}\times\mathfrak{z}\times L^2_{k,G}(M)$,
in a neighborhood of zero.
The restriction $k>m+5$ ensures that the curvature tensors of $g_{t,\alpha,\phi}$ are well defined.

\subsubsection{Nondegeneracy of the relative Futaki invariant}

As the notation in (\ref{eq:rel-Futaki}) suggests the dependence of $\mathcal{F}_{G,\xi_\alpha}$ on
$\xi_\alpha =\xi+\alpha\in\mathfrak{z}^+$ will be important.  In a following section we will compute the derivative
\[ D_g \mathcal{F}_{G,\xi +t\alpha} (\alpha)= \frac{d}{dt} \mathcal{F}_{G,\xi +t\alpha}|_{t=0}.\]

Note that one must be careful that as $\xi_\alpha$ varies in $\mathfrak{z}^+$ one cannot assume that $\hol^T(\xi_\alpha, \ol{J})$
is unchanged because we are changing the foliation.  We assume that our starting structure $(g,\eta,\xi,\Phi)$ has
$G'=\Aut(g,\eta,\xi,\Phi)_0$ a maximal compact subgroup of $\Fol(\mathscr{F}_\xi,\ol{J})$.
We restrict $\mathcal{F}_{G,\xi_\alpha}$ to $\mathfrak{p}/\mathfrak{g}$ and differentiate with respect to $\alpha\in\mathfrak{z}$
at $\alpha=0$ to get
\begin{equation}\label{eq:rel-Futaki-der}
D_g \mathcal{F}_{G,\xi}: \mathfrak{p}/\mathfrak{g}\cong\mathfrak{z}'/\mathfrak{z} \rightarrow\mathfrak{z}^*.
\end{equation}
\begin{defn}
The Futaki invariant relative to $G$ is said to be $G'$-nondegenerate if (\ref{eq:rel-Futaki-der}) is injective.
\end{defn}

\section{Proof of main theorem}

The proof of the main theorem will depend on variation formulae as we vary the Sasaki structure as in (\ref{eq:Sasaki-var1}).

For $X\in\hol^T(\xi,\ol{J})^{1,0}_0$ the normalized potential $f_X \in\mathcal{H}^b_g$ can be written in terms of the Green's
function $G_g$
\[ f_X =-\sqrt{-1}G_g\bigl(\ol{\partial}_b^*(X\contr d\eta)\bigr),\]
though this will not be used.

\subsection{Variation formulae}

We will consider the first order variations of holomorphy potentials and the reduced scalar curvature with respect to
varying the Sasaki structure in (\ref{eq:Sasaki-var1}), more precisely, the derivative with respect to $\phi\in L^2_{k,G}(M)$
or $\alpha\in\mathfrak{z}$.  In the following $D_g$ will denote the derivative with respect to variations of the Sasaki structure
at $(g,\eta,\xi,\Phi)$.  The proof of the first lemma is easy.

\begin{lem}
Let $X\in\hol^T(\xi,\ol{J})^{1,0}_0$ with holomorphy potential $f_X =u_x +\sqrt{-1}v_X$.  Then
\[ D_g f_X (\phi) =X\phi, \]
and if $X=V-\sqrt{-1}\ol{J}V$ real components are $D_g u_X(\phi) = V\phi,\ D_g v_X (\phi)=-\ol{J}V\phi$.

If $\ol{X}=\ol{V}-\sqrt{-1}\ol{J}\ol{V}$ with $V\in\mathfrak{g}'$ and $\alpha\in C_{\mathfrak{g}'}(X)$,
then a non-normalized holomorphy potential
is $f_X =\sqrt{-1}\eta(V)$.  And
\[ D_g f_X (\alpha)=-\sqrt{-1}\eta(\alpha)\eta(V).\]
\end{lem}

\begin{lem}\label{lem:red-scal-var}
Let $(g,\eta,\xi,\Phi)$ be a $G$-invariant Sasaki structure.  The variation of $s^G_g$ in the direction $\phi\in L^2_{k,G}(M)$ is
\begin{equation}\label{eq:red-scal-def-phi}
 D_g s^G_g (\phi) =-2\mathbb{L}_g \phi +(d\phi,ds^G_g ).
\end{equation}
If $s^G_g =0$, then the variation of $s_g^G$ in the direction $\alpha\in C_{\mathfrak{g}'}(\mathfrak{g})$ is given by
\begin{equation}\label{eq:red-scal-def-alp}
 D_g s_g^G (\alpha)=(\mathbb{1} -\pi^G_g)\Bigl(\eta(\alpha)(2s_g -s_0 +2m) -2(m+1)\Delta_b \eta(\alpha)\Bigr)
\end{equation}
\end{lem}
\begin{proof}
The formula (\ref{eq:red-scal-def-phi}) was proved in~\cite{Sim05}.  More precisely, it is proved that
\[ D_g (\pi^G s^G)(\phi) =\partial^{\#}_g (\pi_g^G s_g)\contr\partial\phi=\frac{1}{2}(d\phi,d(\pi_g^G s_g)) \]
from which (\ref{eq:red-scal-def-phi}) follows.

For (\ref{eq:red-scal-def-alp}) we consider the variation of Sasaki structures $\xi_t =\xi+t\alpha$ with
$\alpha\in C_{\mathfrak{g}'}(\xi)$, $\eta_t =\eta(\xi_t)^{-1}\eta$, and fixing the CR structure.
Let $f_t=\eta(\xi_t)^{-1}$, then $\omega^T_t|_D =f_t \omega^T|_D$, which is just a conformal change.
But the calculation of the variation of the curvature is more subtle as the foliation with respect to which
the transverse connection, ${}^t\nabla^T$, is defined is changing.  One calculates, with $X,Y\in\Gamma(D)$
\begin{equation}
\begin{split}
{}^t\nabla_X^T Y =\nabla^T_X Y & +\frac{1}{2}\bigl(d\log f(X)Y +d\log f(Y)X -g^T(X,Y)(d\log f)^{\#} \bigr) \\
										 & +f\bigl(g^T(\Phi(X),Y)\pi_\xi(\xi_t)-g(\xi_t,Y)\Phi(X)-g(\xi_t,X)\Phi(Y)\bigr),
\end{split}
\end{equation}
where $\pi_\xi :TM \rightarrow D$ is the orthogonal projection with respect to $g$ and $f=f_t$.  And
\begin{equation}\label{eq:trans-con-der}
\begin{split}
\dot{\nabla}^T_X Y =\frac{d}{dt}{}^t\nabla^T_X Y|_{t=0}
 			=& \nabla^T_X Y +\frac{1}{2}\bigl(d\dot{f}(X)Y +d\dot{f}(Y)X -g^T(X,Y)(d\dot{f})^{\#} \bigr) \\
						 & +g^T(\Phi(X),Y)\pi_\xi(\alpha)-g(\alpha,Y)\Phi(X)-g(\alpha,X)\Phi(Y)\bigr).
\end{split}
\end{equation}
Using (\ref{eq:trans-con-der}) one calculates, via a long but routine calculation the derivative at $t=0$ of the
transverse curvature $\dot{R}^T(X,Y)Z$.  And contracting the result gives
\begin{equation}
\begin{split}
\frac{d}{dt} s_t^T |_{t=0} & =-\dot{f}s_g^T +2(m+1)\Delta_b \dot{f} \\
						   & = \eta(\alpha)(s_g +2m) -2(m+1)\Delta_B \eta(\alpha).	
\end{split}
\end{equation}

It remains to differentiate $\pi_g^G$.  We let
$\pi^G_t :L^2_{k,G}(M)\rightarrow\sqrt{-1}\mathcal{H}_{g_t}^{\mathfrak{z}}\subset L^2_{k,G}(M)$ be the projection
defined by the above Sasaki structure with $\xi_t =\xi+t\alpha$ and fixed CR structure.
Since $s_g^G =(\mathbb{1}-\pi_g^G)s_g =0$, it is sufficient to compute $(\mathbb{1}-\pi_g^G)(\frac{d}{dt}\pi^G_t|_{t=0})s_g.$
We claim that
\begin{equation}\label{eq:proj-dif-claim}
(\mathbb{1}-\pi_g^G)(\frac{d}{dt}\pi^G_t|_{t=0})s_g =(\mathbb{1}-\pi_g^G)(-\eta(\alpha)(s_g -s_0)).
\end{equation}
We may assume that $\partial_g^{\#}s_g \not\equiv 0$, otherwise both sides vanish.
Let $\{X_0 =\xi, X_1,\ldots,X_r\}$ be a basis of $\mathfrak{z}$ with $X_1$ chosen so that
$\ol{J}\ol{X}_1+\sqrt{-1}\ol{X}_1 =\partial_g^{\#}s_g$ and $\eta(X_1)=s_g -s_0$.  Then
$p^0_t =1, p^1_t =\eta_t(X_1),\ldots, p^r_t =\eta_t(X_r)$ is a basis of $\mathcal{H}_{g_t}^{\mathfrak{z}}$ for
small $t$.  By the Gramm-Schmidt procedure, using the $L^2$ inner product induced by $g$, we obtain an orthonormal
basis $\{f^0_t,\ldots,f^r_t \}$ from $\{p_t^j\}$.  In terms of this basis we have
$\pi^G_t s_g =\sum_{j=0}^r \langle f^j_t,s\rangle_{L^2}f^j_t.$   Thus we have
\begin{equation}\label{eq:diff-proj}
(\mathbb{1}-\pi_g^G)(\frac{d}{dt}\pi^G_t|_{t=0})s_g =(\mathbb{1}-\pi_g^G)\sum_{j=0}^r \langle f^j_0,s\rangle_{L^2}\frac{d}{dt}f^j_t|_{t=0},
\end{equation}
because each $f^j_0$ is in the kernel of $(\mathbb{1}-\pi_g^G)$.

Note that only the $j=0,1$ terms in (\ref{eq:diff-proj}) are possibly non-trivial. 
We have $p^1_t =\eta(\xi +t\alpha)^{-1} \eta(X_1)=\eta(\xi +t\alpha)^{-1}(s_g -s_0)$, and
$f^0_t =(\Vol(g_t))^{-1/2}$.  So we have
\begin{equation}\label{eq:basis-ele}
 f^1_t =\frac{p^1_t -\langle p^1_t,f^0_t \rangle_{\sm L^2} f^0_t}{\| p^1_t -\langle p^1_t,f^0_t \rangle_{\sm L^2} f^0_t\|_{L^2}}.
\end{equation}
Since $\frac{d}{dt} f^0_t|_{t=0}$ is a constant function, (\ref{eq:diff-proj}) is
\[ (\mathbb{1}-\pi_g^G)(\frac{d}{dt}\pi^G_t|_{t=0})s_g =\|s_g -s_0 \|_{L^2} (\mathbb{1}-\pi_g^G)\frac{d}{dt}f^1_t|_{t=0}.\]
Again, because constant functions and $p^1_0$ are annihilated by $ (\mathbb{1}-\pi_g^G)$,
\[ \frac{d}{dt}f^1_t|_{t=0} =\frac{-\eta(\alpha)(s_g -s_0)}{\|s_g -s_0\|_{L^2}} \mod\quad\ker(\mathbb{1}-\pi_g^G),\]
and (\ref{eq:proj-dif-claim}) follows.
\end{proof}

Lemma~\ref{lem:red-scal-var} and (\ref{eq:rel-Futaki-form}) has the following consequence
\begin{prop}\label{prop:rel-Fut-form}
Suppose $(g,\eta,\xi,\Phi)$ be a $G$-invariant Sasaki structure satisfying $s^G_g =0$.  Then the derivative of the relative
Futaki invariant is
\[ D_g \mathcal{F}_{G,\xi,X}(\alpha) =\int_M v_X(\mathbb{1}-\pi_g^G)(\eta(\alpha)(2s_g -s_0 +2m) -2(m+1)\Delta_b \eta(\alpha))\, d\mu_g,  \]
where $X\in\mathfrak{z}'$ has potential $\sqrt{-1}v_X$ and $\alpha\in\mathfrak{z}$.
\end{prop}
This has the following consequence.
\begin{cor}\label{cor:rel-Fut-SE}
Let $(g,\eta,\xi,\Phi)$ be a $G$-invariant Sasaki-Einstein structure.  Then if
$\mathfrak{p}/\mathfrak{g}\cong\mathfrak{z}'/\mathfrak{z}$ nonzero, the relative Futaki invariant is degenerate.
\end{cor}
\begin{proof}
By the Lichnerowicz-Matsushima theorem the Killing potentials $\eta(\alpha)$, when normalized to have zero integral,
$v_\alpha =\eta(\alpha)-\eta(\alpha)_0$ satisfy $\Delta_b v_\alpha =4(m+1)v_\alpha$.  It is then easy to see that
the integrand in Proposition~\ref{prop:rel-Fut-form} vanishes.
\end{proof}

\subsection{Main theorem}

\subsubsection{Proof of main theorem}
A Sasaki metric $g_{t,\alpha,\phi}$ as in (\ref{eq:Sasaki-var1}) with $\phi\in L^2_{k+4,G}(M)$, $k>m+1$, is $G$-invariant.
We have the space of holomorphy potentials $\mathcal{H}^{\mathfrak{g}}_{t,\alpha,\phi}$ for $g_{t,\alpha,\phi}$, and the
subspace of $G$-invariant potentials $\mathcal{H}^{\mathfrak{z}}_{t,\alpha,\phi} \subseteq\mathcal{H}^{\mathfrak{g}}_{t,\alpha,\phi}$.
Using the metric $g_{t,\alpha,\phi}$ to define the $L^2$ inner product on $\phi\in L^2_{k,G}(M)$ we have the orthogonal
decomposition
\[ L^2_{k,G}(M)=\sqrt{-1}\mathcal{H}^{\mathfrak{z}}_{t,\alpha,\phi} \oplus W_{k,t,\alpha,\phi}, \]
and the projections
\[ \pi^G_{t,\alpha,\phi} :L^2_{k,G}(M)\rightarrow\sqrt{-1}\mathcal{H}^{\mathfrak{z}}_{t,\alpha,\phi},\text{  and  }
\pi^W_{t,\alpha,\phi} :L^2_{k,G}(M)\rightarrow W_{k,t,\alpha,\phi}. \]
The reduced scalar curvature of $g_{t,\alpha,\phi}$ is given by
\begin{equation}
s^G_{t,\alpha,\phi} =\pi^W_{t,\alpha,\phi}(s_{t,\alpha,\phi})=(\mathbb{1} -\pi^G_{t,\alpha,\phi})(s_{t,\alpha,\phi})
\end{equation}
We are looking for solutions of the equation
\begin{equation}
s^G_{t,\alpha,\phi} =0.
\end{equation}
For $(t,\alpha,\phi)\in\mathcal{U}\subset\mathcal{B}\times\mathfrak{z}\times L^2_{k+4,G}(M)$ with $\mathcal{U}$ a small enough
neighborhood of $(0,0,0)$ the associated Sasaki structure (\ref{eq:Sasaki-var1}) is well defined and
$\pi^W_{0} : W_{k,t,\alpha,\phi}\rightarrow W_{k,o}$ is an isomorphism.

Let $\mathcal{V}=\mathcal{U}\cap \mathcal{B}\times\mathfrak{z}\times W_{k+4,0}$.  Then we define a map
\begin{equation}\label{eq:def-map}
\begin{array}{rccc}
\mathscr{S}: & \mathcal{V} & \rightarrow &  \mathcal{B}\times W_{k,0}\\
			 & (t,\alpha,\phi) & \mapsto & \Bigl( t, \pi^W_{0}(s^G_{t,\alpha,\phi}) \Bigr)
\end{array}
\end{equation}

\begin{lem}\label{lem:def-map-der}
The map $\mathscr{S}$ is $C^1$ and its differential is Fredholm.  Assume that the Sasaki structure $(g,\eta,\xi,\Phi)$ at
$(0,0,0)\in\mathcal{V}$ has vanishing reduced scalar curvature $s^G_g =0$, then the differential of $\mathscr{S}$ at
$(t,\alpha,\phi)=(0,0,0)$ is
\[\left[
\begin{array}{ccc}
\mathbb{1} &    0 &  0    \\
*          & \multicolumn{2}{c}{\mathcal{S}_g^G}
\end{array}\right]\]
where
\begin{equation}\label{eq:def-map-der}
\begin{split}
\mathcal{S}_g^G (\dot{\alpha},\dot{\phi}) & = -2\mathbb{L}_g \dot{\phi}+\pi^W_{0}(\dot{s}^G_g)(\alpha)) \\
										  &	=-2\mathbb{L}_g \dot{\phi} +\pi^W_{0}\bigl(\eta(\alpha)(2s_g -s_0 +2m) -2(m+1)\Delta_b \eta(\alpha))\bigr).
\end{split}
\end{equation}
\end{lem}
\begin{proof}
Since the reduced scalar curvature $s^G_{t,\alpha,\phi}$ is $C^1$ in $(t,\alpha,\phi)$, the map $\mathscr{S}$ is $C^1$.
The formula for $\mathcal{S}_g^G (\dot{\alpha},\dot{\phi})$ follows from Lemma~\ref{lem:red-scal-var}.
\end{proof}

\begin{prop}\label{prop:def-map-sub}
Suppose that $(g,\eta,\xi,\Phi)$ at $(0,0,0)\in\mathcal{V}$ had vanishing reduced scalar curvature $s^G_g =0$, then
the Fr\'echet derivative of $\mathscr{S}$, defined in (\ref{eq:def-map}), at $(t,\alpha,\phi)=(0,0,0)$ has index $\dim\mathfrak{z}$ and
is an submersion if and only if the relative Futaki invariant $\mathcal{F}_{G,\xi}$ is nondegenerate at $\xi$.
\end{prop}
\begin{proof}
Note that $D_g \mathscr{S}$ is a compact perturbation of
\[\mathcal{B}\times\mathfrak{z}\times W_{k+4,0} \ni(s,\dot{\alpha},\dot{\phi})\mapsto(s,-2\mathbb{L}_g \dot{\phi})\in \mathcal{B}\times W_{k,0}.\]
Since the index of $\mathbb{L}_g :W_{k+4,0} \rightarrow W_{k,0}$ is zero, the index of $D_g \mathscr{S}$ must be
$\dim\zeta$.

If $D_g \mathscr{S}$ is not surjective, there is a $\psi\in W_{k,0}$ in the cokernel.
We have from (\ref{eq:def-map-der})
\[ \langle\mathbb{L}_g \dot{\phi},\psi\rangle =0,\quad\text{and}\quad\langle\pi^W_{0}(\dot{s}^G_g)(\alpha)),\psi\rangle=0,\]
for all $\dot{\phi}\in W_{k+4,0}$.  The first equation implies $\psi\in\sqrt{-1}\mathcal{H}^{\mathfrak{z}'}_g$,
so $X=\ol{J}\grad\psi \in\mathfrak{z}'$.  Thus
\begin{equation}
\int_M \psi\pi^W_{0}(\dot{s}_g^G(\dot{\alpha}))\, d\mu_g =\int_M \psi\dot{s}_g^G(\dot{\alpha})\, d\mu_g =D_g \mathfrak{F}_{G,\xi,X}(\dot{\alpha}),
\end{equation}
where the second equality uses (\ref{eq:rel-Futaki-form}) and that $s_g^G =0$.
If $\mathcal{F}_{G,\xi}$ is nondegenerate, then this implies $X\in\mathfrak{z}$ and $\psi\in\sqrt{-1}\mathcal{H}^{\mathfrak{z}}_g$
contradicting $\psi\in W_{k,0}$.
\end{proof}

\begin{thm}\label{thm:mainthm}
Let $(\mathscr{F}_\xi,\ol{J}_t),\ t\in\mathcal{B}$, be a $G$-equivariant $(1,1)$-type deformation with $\mathcal{B}$ smooth and 
fixing the smooth structure of $\mathscr{F}_\xi$, where $G$ is a compact connected group with $\xi\in\mathfrak{g}$.
Suppose $(g,\eta,\xi,\Phi)\in\mathcal{S}(\xi,\ol{J}_0)$ has vanishing reduced scalar curvature $s^G_g =0$.
If the relative Futaki invariant $\mathcal{F}_{G,\xi}$ is nondegenerate at $\xi$, then there is a neighborhood $\mathcal{V}$
of $(0,0,0)\in\mathcal{B}\times\mathfrak{z}\times C_b^\infty(M)^G$ so that
\[ \mathcal{E} =\{ (t,\alpha, \phi)\in\mathcal{V} : (g_{t,\alpha,\phi},\eta_{t,\alpha,\phi},\xi+\alpha,\Phi_{t,\alpha,\phi})\text{  has } s^G_{t,\alpha,\phi} =0\}, \]
is a smooth manifold of dimension $\dim_R \mathcal{B} +\dim_R \mathfrak{z}$.

Furthermore, the map $\varpi:\mathcal{E}\rightarrow\mathcal{B},\ \varpi(t,\alpha,\phi)=t$ is a submersion with fibers of dimension
$\dim_R \mathfrak{z}$.  And any $(t,\alpha, \phi))\in\mathcal{E}$ has $\phi\in C_b^\infty(M)^G$.
\end{thm}
\begin{proof}
By Proposition~\ref{prop:def-map-sub} the map (\ref{eq:def-map-der}) is a submersion at $(0,0,0)$.
Let $K=\ker D_g \mathscr{S}\subset T_0\mathcal{B}\times\mathfrak{z}\times W_{k+4,0}$.  We identify $\mathcal{B}\subset T_0\mathcal{B}$
as a Euclidean space.  Let $\pi:\mathcal{V}=\mathcal{U}\cap \mathcal{B}\times\mathfrak{z}\times W_{k+4,0} \rightarrow K$ be
the orthogonal projection.  Then the differential at zero of
\[ \mathscr{S} \times\pi :\mathcal{V} \rightarrow \mathcal{B}\times W_{k,0} \times K,\]
is an isomorphism.  The inverse function theorem provides an inverse, and
$\mathcal{B}\times K\ni (t,s)\mapsto(\mathscr{S} \times\pi)^{-1}(t,0,s)$,
parametrizes $\mathcal{E}$.

$\varpi$ is a submersion because $S$ is orthogonal to $\mathcal{B}\times\{0\}\times\{0\}$.
If $(t,\alpha, \phi)\in\mathcal{E}$, then $g^T_{t,\alpha, \phi}$ is transversely extremal.  The regularity result
of~\cite{LebSim93}, applied in a local foliation chart, shows that $\phi\in C_b^\infty(M)^G$.
\end{proof}

\subsubsection{Maximal torus case}
The case in which $G=T^r \subseteq G'$ is a maximal torus gives a somewhat stronger result than in general.
Furthermore, it is easier to find examples, because the nondegeneracy of the Futaki invariant holds trivially.
In this section $G=T^r$ is a maximal torus in the maximal compact subgroup $G'\subset\Fol(M,\mathscr{F}_\xi,\ol{J})$.
Note that we have $\mathfrak{z}=\mathfrak{z}'=\mathfrak{g}$.

The proof of the following is obvious.
\begin{lem}
Suppose $G' \subset\Fol(M,\mathscr{F}_\xi,\ol{J})$ is maximal compact and $G=T^r \subseteq G'$ is a maximal torus.
Then $\mathfrak{p}/\mathfrak{g} =\mathfrak{z}'/\mathfrak{z} =0$.
\end{lem}

\begin{cor}\label{cor:maincor-tor}
Let $(\mathscr{F}_\xi,\ol{J}_t),\ t\in\mathcal{B}$, be a $G$-equivariant deformation of $(1,1)$-type with $\mathcal{B}$ smooth and fixing the smooth structure of $\mathscr{F}_\xi$, where $G=T^r$ is a maximal torus of $G'=\Aut(g,\eta,\xi,\Phi)_0$, and suppose
$(g,\eta,\xi,\Phi)\in\mathcal{S}(\xi,\ol{J}_0)$ has vanishing reduced scalar curvature $s^G_g =0$.
Then there is a neighborhood of zero
$\mathcal{V}\subset\mathfrak{B}\times\mathfrak{g}$ so that for $(t,\alpha)\in\mathcal{V}$ there is smooth Sasaki metric $g_{t,\alpha,\phi_{t,\alpha}}$ satisfying $s^G_{g_{t,\alpha,\phi_{t,\alpha}}} =0$.  So that for each fixed $t\in\mathcal{B}$
close to zero, the space of extremal metrics is parametrized by a neighborhood of zero in $\mathfrak{g}$.
\end{cor}
\begin{proof}
By the Lemma the relative Futaki invariant is nondegenerate.  As above, define $K=\ker D_g \mathscr{S}$.
Suppose $(0,0,\dot{\phi})\in K$, then by (\ref{eq:def-map-der}) we have $\mathbb{L}_g \dot{\phi} =0$.
So $\dot{\phi}\in\sqrt{-1}\mathcal{H}^{\mathfrak{z}'}_g =\mathcal{H}^{\mathfrak{z}}_g$, and $\dot{\phi}=0$ since
$\dot{\phi}\in W_{k+4,0}$.  Therefore the projection $\varpi: K\rightarrow \mathfrak{g},\ \varpi(t,\alpha,\phi)=\alpha$
is an isomorphism.  We consider the map
\[\mathscr{S} \times\varpi\circ\pi :\mathcal{V} \rightarrow \mathcal{B}\times W_{k,0} \times\mathfrak{g},\]
whose differential at zero is an isomorphism.  The proof then follows from the inverse function theorem as in
Theorem~\ref{thm:mainthm}.
\end{proof}

\subsection{Sasaki-Einstein case}\label{subsec:Sasaki-Ein}

\subsubsection{Necessary condition for a Sasaki-Einstein structure}
Because of Corollary~\ref{cor:rel-Fut-SE} we might as well assume $G=T^r$ and $T^r \subseteq G'$ is a maximal torus.
We recall the necessary condition for $(g,\eta,\xi,\Phi)$ to admit a transverse K\"{a}hler deformation to a Sasaki-Einstein
structure, or rather a structure which is transversally K\"{a}hler-Einstein, $\Ric^T =\tau g^T,\ \tau>0$.
The following necessary conditions are well known.  See~\cite{FutOnWan09} or~\cite{MarSpaYau08}.
\begin{prop}\label{prop:CY-cond}
The following conditions are equivalent.
\begin{thmlist}
\item  $\frac{\tau}{2\pi}[\omega^T] =c^b_1(\mathscr{F}_\xi)$ in $H^2_b(M/\mathscr{F}_\xi)$ for $\tau>0$. \label{prop:SE-item1}

\item  The class $c_1^b(\mathscr{F}_\xi)$ is represented by a positive $(1,1)$ basic form and $c_1(D)=0$. \label{prop:SE-item2}

\item  There exists a nowhere vanishing holomorphic $(m+1,0)$-form $\Omega\in\Gamma(\Lambda^{1,1} C(M))$ for which
$\mathcal{L}_{\xi} \Omega =\sqrt{-1}\frac{\tau}{2} \Omega$.\label{prop:SE-item3}
If $M$ is not simply connected, then we may have to take $\Omega$ to be multi-valued, or
$\Omega\in\Gamma(\Lambda^{1,1} C(M))^{\otimes\ell}$.
\end{thmlist}
\end{prop}
\begin{rmk}
Note that the conditions imply $\pi_1(M)$ must be finite.  The transverse Calabi-Yau theorem~\cite{Kac90} implies the existence
of a transverse K\"{a}hler deformation to transversal metric with $\Ric^T >0$.  After a possible homothety, this lifts to a Sasaki structure with $\Ric_g >0$, and the claim follows from Myers' Theorem.
\end{rmk}
\begin{proof}
In order to prove the equivalence of (\ref{prop:SE-item1}) and (\ref{prop:SE-item2}) consider
the Gysin sequence~\cite[Ch. 7]{BoyGal08}
\begin{equation}\label{eq:Gysin-seq}
 0\rightarrow H^0_b(M/\mathscr{F}_\xi)\overset{\delta}{\longrightarrow}H^{2}_b(M/\mathscr{F}_\xi)\overset{\iota}{\longrightarrow}H^2(M,\R)
\rightarrow\cdots,
\end{equation}
where $\delta\alpha =[\alpha d\eta]_b$.  If we have (\ref{prop:SE-item1}), then $\iota(c^b_1(\mathscr{F}_\xi))=0$.  But this 
represents $c_2(D)$.  If (\ref{prop:SE-item2}) holds, then again $\iota(c^b_1(\mathscr{F}_\xi))=0$, so
there exists an $\alpha\in\R$ with $\delta(\alpha)=2\alpha[\omega^T]=c^b_1(\mathscr{F}_\xi$.  But by assumption we must have
$\alpha>0$.

Supposing (\ref{prop:SE-item3}) we have
\begin{equation}\label{eq:hol-form}
\Bigl(\frac{\sqrt{-1}}{2}\Bigr)^{m+1} (-1)^{m(m+1)/2} \Omega\wedge\ol{\Omega}=\exp(h)\frac{1}{(m+1)!} \omega^{m+1},
\end{equation}
with $\omega$ the K\"{a}hler form of $(C(M),\ol{g})$ and $h\in C^\infty(C(M))$.  Taking the Lie derivative $\mathcal{L}_{\xi}$ of
(\ref{eq:hol-form}), we see the condition in (\ref{prop:SE-item3}) implies $\mathcal{L}_{\xi} h =0$.
We make a homothetic deformation $(g_a,\eta_a,\xi_a,\Phi)$ of $(g,\eta,\xi,\Phi)$ with $a=\frac{\tau}{2m+2}$, i.e.
$\eta_a =a\eta,\ \xi_a =\frac{1}{a}\xi$ and $g_a =ag +(a^2 -a)\eta\otimes\eta$.
We use our original notation for this homothetic Sasaki structure,
then we have
$\mathcal{L}_{\xi} \Omega =\sqrt{-1}(m+1)\Omega$.  Then applying $\mathcal{L}_{r \partial_{r}}$ to (\ref{eq:hol-form}),
with the new Sasaki structure, since $\mathcal{L}_{r\partial_{r}}\omega =2\omega$,
we have $\mathcal{L}_{r\partial_{r}} h =0$.  Thus $h \in C_b^\infty(M)$ is basic, and the Ricci form $\rho$ of
$(C(M),\ol{g})$ is
\begin{equation}\label{eq:Ricci-cone}
\sqrt{-1}\partial\ol{\partial} h=\rho =\rho^T -(2m+2)\omega^T,
\end{equation}
which implies (\ref{prop:SE-item1}) with $\tau=2m+2$.

Conversely, assuming (\ref{prop:SE-item1}) we make a homothetic transformation so that $\tau =2m+2$.  Then
the basic cohomology class $[\rho^T -(2m+2)\omega^T]_b =0$, so the transverse $\partial\ol{\partial}$-Lemma~\cite{Kac90} gives
an $h\in C^\infty(M)_b$ satisfying (\ref{eq:Ricci-cone}).  Define an Hermitian metric on $\Lambda^{1,1} C(M)$ by
\begin{equation}\label{eq:Herm-met}
\|\Omega\|_h := \Bigl(\frac{\sqrt{-1}}{2}\Bigr)^{m+1}(m+1)! (-1)^{m(m+1)/2}\exp(-h)\frac{\Omega\wedge\ol{\Omega}}{\omega^{m+1}}.
\end{equation}
The curvature of the Chern connection of $\|\cdot\|_h$ is $\sqrt{-1}\partial\ol{\partial} h -\rho =0$.  Therefore the universal
cover $\varpi:\tilde{M}\rightarrow M$ has a parallel section $\Omega\in\Lambda^{1,1} C(\tilde{M})$.
\end{proof}

Suppose Proposition~\ref{prop:CY-cond} holds for $(g,\eta,\xi,\Phi)$ with $G\subseteq\Aut(g,\eta,\xi,\Phi)_0$.
We assume $\tau=2m+2$ for simplicity.  Since $G$ is connected (\ref{prop:SE-item1}) is preserved by $G$, and in
(\ref{eq:Ricci-cone}) we may take $h\in C_b^\infty(M)$ to be $G$-invariant.  Thus the metric $\|\cdot\|_h$ is $G$-invariant.
Because $g^*\Omega$ is parallel and $\|g^*\Omega\|_{h} =1$ for $g\in G,\ g^*\Omega =\chi(g)\Omega$ with $\chi(g) \in\U(1)$.
And
\begin{equation}\label{eq:hol-form-char}
 \chi: G\rightarrow\U(1),
\end{equation}
is a character.

For the remainder of the section we suppose that $G=T\subset G'=\Aut(g,\eta,\xi,\Phi)$ is a maximal torus.
Then of course, $\mathfrak{z}=\mathfrak{z}'=\mathfrak{g}$.
\begin{defn}\label{defn:char-hyper}
We define the \emph{characteristic hyperplane} of a Sasaki structure $(g,\eta,\xi,\Phi)$ satisfying Proposition~\ref{prop:CY-cond}
to be the hyperplane $\mathcal{P}=\{X\in\mathfrak{g} : \chi_* X =\sqrt{-1}(m+1) \}\subset\mathfrak{g}$ containing $\xi$.

And define $\mathcal{Q}=\mathcal{P}-\xi=\{X\in\mathfrak{g} : \ker\chi_* \}\subset\mathfrak{g}$ to be the corresponding linear space.
\end{defn}

For any $\xi_\alpha =\xi+\alpha \in\mathcal{P}\cap\mathfrak{g}^+$ the Sasaki structure
$(g_\alpha,\eta_\alpha,\xi_\alpha,\Phi_\alpha)$
with Reeb vector field $\xi_\alpha$ defined in (\ref{eq:Sasaki-var1}) and (\ref{eq:Sasaki-var2}) satisfies Proposition~\ref{prop:CY-cond}.

\subsubsection{Volume functional and Futaki invariant}
We will consider the space of Sasaki structures on $M$ considered in Section~\ref{subsubsec:Sasak-cone} depending on
$(\xi_\alpha,\phi)\in\mathfrak{z}^+ \times C^\infty(M)^G$ with Reeb vector field
$\xi_\alpha=\xi +\alpha \in\mathfrak{z}^+$ and
$\eta_{\alpha,\phi}=\eta_{\xi_\alpha} +d^c \phi$.  These Sasaki structures correspond to a space of K\"{a}hler cone
metrics on $(C(M),I)$ with $G$ contained in the isometry group.  We denote this space of Sasaki structures on $M$ by
$\mathcal{S}(G,I)$.  Just as in~\cite{MarSpaYau08} we consider the volume functional
\begin{equation}\label{eq:Vol-funct}
\begin{array}{rcl}
\mathcal{S}(G,I) & \overset{\Vol}{\longrightarrow} & \R \\
(g_{\alpha,\phi},\eta_{\alpha,\phi},\zeta_\alpha,\Phi_{\alpha,\phi}) & \mapsto & \int_M d\mu_{g_{\alpha,\phi}},
\end{array}
\end{equation}
where it can be shown that $\Vol(g_{\alpha,\phi}) =\frac{1}{m!}\int_M \eta_{\alpha,\phi}\wedge(\frac{1}{2}d\eta_{\alpha,\phi})^m$
depends only on the Reeb vector field $\xi_\alpha$.  See also~\cite{FutOnWan09}.  Thus (\ref{eq:Vol-funct}) defines a
functional
\begin{equation}
\Vol:\mathfrak{g}^+ \rightarrow\R.
\end{equation}

We will need the first and second variation formulae of $\Vol$ which were first calculated in~\cite{MarSpaYau08}.  Let
$\{(g(t),\eta_t,\xi_t,\Phi_t\}_{-\epsilon<t<\epsilon}$ be a 1-parameter family of Sasaki structures in
$\mathcal{S}(G,I)$ with $g(0)=g$ and $\dot{\xi}_0 =X$, then
\begin{equation}\label{eq:Vol-1st-der}
D_g \Vol(X)=\frac{d}{dt} \Vol(g(t))|_{t=0} =-(m+1)\int_M \eta(X)\, d\mu_g.
\end{equation}
For the second derivative, let $\{(g(t),\eta_t,\xi_t,\Phi_t\}_{-\epsilon<t<\epsilon}$ be a 1-parameter family of Sasaki structures in
$\mathcal{S}(G,I)$ with $g(0)=g$ and $\dot{\xi}_0 =Y$, then
\begin{equation}\label{eq:Vol-2nd-der}
\begin{split}
D^2_g \Vol(X,Y) & =-(m+1)\frac{d}{dt}\Bigl(\int_M \eta_t(X)\, d\mu_{g_t} \Bigr)|_{t=0} \\
                & =(m+1)(m+2) \int_M \eta(X)\eta(Y)\, d\mu_g.
\end{split}
\end{equation}

Therefore, $\Vol:\mathfrak{g}^+ \cap\mathcal{P} \rightarrow\R$ is strictly convex function on a convex polytope
$\mathfrak{g}^+ \cap\mathcal{P}$.  Moreover, one can show that the integral in (\ref{eq:Vol-funct}) goes to infinity
as $\xi_\alpha$ approaches the boundary of $\mathcal{C}_{\mathfrak{z}}=\mathfrak{g}^+$.

\begin{prop}\label{prop:Vol-Fut}
Let $(g',\eta',\xi',\Phi')\in\mathcal{S}(G,I)$ have $\xi'\in\mathcal{P}$.  Then $\xi'$ is a critical point
for $\Vol:\mathfrak{g}^+ \cap\mathcal{P} \rightarrow\R$ if and only if the Futaki invariant restricted to $\mathfrak{g}$
vanishes, $\mathcal{F}_{\xi'}|_{\mathfrak{g}}\equiv 0$.
\end{prop}
\begin{proof}
We consider the following set of potentials for the transversely holomorphic vector fields
$\pi(\mathfrak{g})^{1,0}\subseteq\hol^T(\xi',\ol{J})_0$.  Define
\[ \tilde{\mathcal{H}}_{g'} :=\{ \sqrt{-1}\eta'(X)\, |\, X\in\mathcal{Q}\}\subset\mathcal{H}_{g'}^{\mathfrak{g}}. \]

We define the operator appearing in~\cite{Fut88}, with $h$ given in (\ref{eq:Ricci-cone}) and $\Box_b =\frac{1}{2}\Delta_b$ the
complex Laplacian,
\begin{equation}\label{eq:mod-Laplac}
\Box^h_b u:=\Box_b u-\partial^{\#}u\contr\partial{h}.
\end{equation}
Note that $\Box^h_b$ is self adjoint with respect to a weighted volume on $M$,
\[ \int_M \Box^h_b u \overline{v} e^h d\mu_{g'} =\int_M u\overline{\Box^h_b v} e^h d\mu_{g'}  \]

We say that a holomorphy potential $u_X$ is \emph{normalized} if
\begin{equation}
\int_M u_X\, e^h d\mu_{g'}=0.
\end{equation}

We will need the following result from~\cite{Fut88} and~\cite{FutOnWan09}.
\begin{thm}
Suppose $(g',\eta',\xi',\Phi')$ satisfies Proposition~\ref{prop:CY-cond} and $\Box^h_b$ defined in (\ref{eq:mod-Laplac}).
The eigenspace $\{ u\in C_b^\infty(M,\C)\, |\, \Box^h_b u =(2m+2)u \}$ is isomorphic to the space of normalized
holomorphy potentials.
\end{thm}

Let $X\in\mathcal{Q}$, and apply $\mathcal{L}_{IX}$ to (\ref{eq:hol-form}) to get
\begin{equation}
\begin{split}
0 & = (IX)h + \frac{1}{2}\Delta^{C(M)} r^2 \eta'(X) \\
   & = (IX)h + \frac{1}{2}\Bigl(\frac{1}{r^2}\Delta^M_b -\frac{\partial^2}{\partial r^2} -\frac{(2m+1)}{r}\frac{\partial}{\partial r} \Bigr)r^2 \eta'(X) \\
   & = (IX)h +\frac{1}{2}\Delta^M_b \eta'(X) -(2m+2)\eta'(X) \\
   & = \Box^h_b u -(2m+2)u.
\end{split}
\end{equation}
And it follows that the space of normalized holomorphy potentials for $\pi(\mathfrak{g})^{1,0}\subseteq\hol^T(\xi',\ol{J})_0$ is $\tilde{\mathcal{H}}_{g'}$.

We have $\frac{1}{2}\ol{J}\grad\eta'(X) = X$ and
\begin{equation}
\begin{split}
\mathcal{F}_{\xi'} (X) & = \int_M d^c h(X)\, d\mu_{g'} \\
                       & = \int_M -\ol{J}X h -\frac{1}{2}\Delta_b \eta'(X)\, d\mu_{g'} \\
                       & = -(2m+2)\int_M \eta'(X)\, d\mu_{g'},
\end{split}
\end{equation}
from which the Proposition follows.
\end{proof}

\subsubsection{Deformations of Sasaki-Einstein structures}
We now consider a $G$-equivariant deformation $(\mathscr{F}_\xi ,\ol{J}_t)_{t\in\mathcal{B}}$.
It is useful that the Kuranishi space of Section~\ref{subsubsec:Kurani} is always smooth because of the following.
\begin{prop}\label{prop:Kuran-obst}
Suppose that the conditions of Proposition~\ref{prop:CY-cond} hold, or more generally, $c_1^b >0$.
Then $H^2_{\ol{\partial}_b}(\mathcal{A}^{0,\bullet}) =0$.  Thus the Kuranishi space $\mathcal{V}$ and the submanifold of
$G$-equivariant deformations $\mathcal{V}^G$ are smooth.  
\end{prop}
\begin{proof}
Using harmonic theory for the Laplacian $\Delta_{\ol{\partial}_b} =\ol{\partial}_b \ol{\partial}_b^* +\ol{\partial}_b^* \ol{\partial}_b$
associated with the complex (\ref{eq:Dol-comp}) one can prove Serre duality using the same arguments as in~\cite{GriHar78}
to get
\begin{equation}\label{eq:Serre-dual}
\begin{split}
H^2_{\ol{\partial}_b}(\mathcal{A}^{0,\bullet}) & \cong H^{m-2}_{\ol{\partial}_b}(\mathcal{A}^{m,\bullet}\otimes\Lambda^{1,0}_b) \\
											   & \cong H^{m-2}_{\ol{\partial}_b}(\mathcal{A}^{1,\bullet}\otimes\Lambda^{m,0}_b) \\
\end{split}
\end{equation}
By our assumption $\Lambda^{m,0}_b$ admits a connection with negative curvature.  Again using harmonic representatives, the
proof of Kodaira-Nakano vanishing in~\cite{GriHar78} works in this situation and we get the last term in (\ref{eq:Serre-dual}) is
zero because $(m-2) +1 < m$.
\end{proof}

Recall that by Proposition~\ref{prop:obst-pos} the existence of Sasaki structures on a deformation in this case is unobstructed.
We have a family $(g_t,\eta_t,\xi,\Phi_t)\in\mathcal{S}(\xi,\ol{J}_t),\ t\in\mathcal{B}$ of compatible Sasaki structures with
$(g_0,\eta_0,\xi_0,\Phi_0)$ satisfying Proposition~\ref{prop:CY-cond}, with $\tau=2m+2$.  Since $c_1^b(\mathscr{F}_\xi,\ol{J}_t)$ is
unchanged under deformation of $\ol{J}_t$, condition (\ref{prop:SE-item1}) of the Proposition holds for all $t\in\mathcal{B}$.
We can define $h_t \in C^\infty_b(M)$ depending smoothly on $t\in\mathcal{B}$ by
\[ h_t =2G_{g_t^T}(\omega^T_t ,\rho_t^T -(2m+2)\omega_t^T)_{g_t^T} =2G_{g_t^T}(s_t^T -s_0^T),\]
where $G_{g_t^T}$ is the Green's function of $g^T_t$.  By taking parallel displacement from a fixed point with respect to the flat
Chern connection of $\|\,\cdot\,\|_{h_t}$ on $\Lambda^{m+1,0}\bigl(C(M)\bigr)$, we get a smooth family of holomorphic
$(m+1,0)$-forms $\Omega_t$ on the family of cones $C(M_t)=(C(M),I_t)$.  Then for each $t\in\mathcal{B}$ as in (\ref{eq:hol-form-char})
we have a character $\chi_t :G\rightarrow\U(1)$.  Since the characters on $G$ is discrete lattice, $\chi_t$ is independent of
$t\in\mathcal{B}$.  It follows that the characteristic hyperplane $\mathcal{P}\subset\mathfrak{g}$ is independent of
$t\in\mathcal{B}$.

\begin{cor}\label{cor:maincor-Einst}
Let $(g,\eta,\xi,\Phi)$ be a Sasaki-Einstein structure, and suppose that $(\mathscr{F}_{\xi}, \ol{J}_t)_{t\in\mathcal{B}}$ is a
$G$-equivariant deformation, where $G\subseteq G'=\Aut(g,\eta,\xi,\Phi)_0$ is a maximal torus. Then there is a neighborhood
$U\subset\mathcal{B}$ so that for $t\in U$ there is a unique $\alpha_t \in\mathcal{Q}\subset\mathfrak{g}$ and a
$\phi_t \in C^\infty(M)^G$ so that $g_{t,\alpha_t,\phi_t}$ is Sasaki-Einstein.
\end{cor}
\begin{proof}
We modify the map (\ref{eq:def-map})
Let $\tilde{\mathcal{V}}=\mathcal{U}\cap \mathcal{B}\times\mathcal{Q}\times W_{k+4,0}$.  Then we define a map
\begin{equation}\label{eq:def-map-Einst}
\begin{array}{rccc}
\tilde{\mathscr{S}}: & \tilde{\mathcal{V}} & \rightarrow &  \mathcal{B}\times\mathcal{Q}^* \times W_{k,0}\\
			 & (t,\alpha,\phi) & \mapsto & \Bigl( t,\mathcal{F}(\alpha), \pi^W_{0}(s^G_{t,\alpha,\phi}) \Bigr),
\end{array}
\end{equation}
where $\mathcal{F}(\alpha)\in \mathcal{Q}^*$ is defined by
$\mathcal{F}(\alpha)(X):=-\int_M \eta_{\alpha,\phi}(X)\, d\mu_{g_{\alpha,\phi}}$.  Then $D_g \tilde{\mathscr{S}}$ is given
by Lemma~\ref{lem:def-map-der} with the exception of $D_g \mathcal{F}(\dot{\alpha})$ which is given by (\ref{eq:Vol-2nd-der})
\[ D_g \mathcal{F}(\dot{\alpha})(X)=(m+2)\int_M \eta(\dot{\alpha})\eta(X)\, d\mu_g,\quad X\in\mathcal{Q}.\]
It is routine to check that
\[ D_g \tilde{\mathscr{S}}:\mathcal{B}\times\mathcal{Q}\times W_{k+4,0}\rightarrow\mathcal{B}\times\mathcal{Q}^* \times W_{k,0} \]
is an isomorphism.  By the inverse function theorem there is a neighborhood
$\mathcal{U}\subset\mathcal{B}\times\mathcal{Q}^* \times W_{k,0}$ on which $\tilde{\mathscr{S}}^{-1}$ is defined.
Then with $U=\mathcal{U}\cap\mathcal{B}\times\{0\}\times\{0\}$ we set $(t,\alpha_t,\phi_t) =\tilde{\mathscr{S}}^{-1}(t,0,0)$
for $t\in U$.  And $g_{t,\alpha_t,\phi_t}$ is a Sasaki-extremal metric.  Since $s^G_{g_{t,\alpha_t,\phi_t}}=0$ for $t\in U$
we have $\ol{J}_t\grad s_{g_{t,\alpha_t,\phi_t}} \in\pi(\mathfrak{g})\subset\hol^T(\xi+\alpha_t,\ol{J}_t)_0$.

We denote the metric $g_{t,\alpha_t,\phi_t}$ by $g_t$ for brevity.  Let $h_t\in C^\infty_b(M)^G$ satisfy (\ref{eq:Ricci-cone}),
then $\Delta_b h_t =s^T_{g_t} -s^T_0 =s_{g_t} -s_0$, where $s^T_0 =(2m+2)(2m)$ and $s_0$ are the averages of $s^T_{g_t}$ and $s_{g_t}$.
By Proposition~\ref{prop:Vol-Fut} the Futaki invariant $\mathcal{F}_{\xi+\alpha_t} |_{\mathfrak{g}} \equiv 0$, and with
$X=\ol{J}\grad s_{g_t}$ we have
\[\begin{split}
0=\mathcal{F}_{\xi+\alpha_t} (X) & =\int_M d^c h_t(X)\, d\mu_{g_t} \\
                                 & =\int_M (ds_{g_t},dh_t)\, d\mu_{g_t} \\
                                 & =\int_M (s_{g_t},\Delta_b h_t)\, d\mu_{g_t} \\
                                 & =\int_M (s_{g_t}, s_{g_t} -s_0)\, d\mu_{g_t} \\
                                 & =\int_M \|s_{g_t} -s_0\|^2\, d\mu_{g_t}.
\end{split}\]
So $s_{g_t} -s_0 =0$ and $h_t$ is constant, therefore $g_{t,\alpha_t,\phi_t}$ is Sasaki-Einstein.
\end{proof}

\section{Examples}

We describe a family of examples of 7-manifolds on which we can apply Corollary~\ref{cor:maincor-tor} and Corollary~\ref{cor:maincor-Einst}
to give new families of Sasaki-extremal and Sasaki-Einstein metrics.  More details will appear in~\cite{vanC12}.
These examples are deformations of 3-Sasaki manifolds that first appeared in the work of C. Boyer, K. Galicki, B. Mann, and
E. Reese~\cite{BGMR98}.

\begin{defn}\label{eq:3Sasaki}
A Riemannian manifold $(M,g)$ is \emph{3-Sasaki} if the metric cone $(C(M),\ol{g})$ is hyperk\"{a}hler, i.e. $\ol{g}$ admits compatible almost complex structures
$J_\alpha,\ \alpha=1,2,3 $ such that $(C(M),\ol{g},J_1,J_2,J_3)$ is a hyperk\"{a}hler structure.  Equivalently,\\ $\Hol(C(M))\subseteq\Sp(m)$.
\end{defn}

A consequence of the definition is that $(M,g)$ is equipped with three Sasaki structures $(\xi_i,\eta_i,\phi_i),\ i=1,2,3$.
The Reeb vector fields $\xi_k,\ k=1,2,3$ are orthogonal and satisfy $[\xi_i,\xi_j]=2\varepsilon^{ijk}\xi_k$,
where $\varepsilon^{ijk}$ is anti-symmetric in the indices $i,j,k \in\{1,2,3\}$ and $\varepsilon^{123}=1$.

The Reeb vector fields $\xi_k,\ k=1,2,3$ generate an action of $\Sp(1)$ or $\SO(3)$.
A 3-Sasaki manifold $M$ comes with a family of related geometries.  The maps are labeled with their generic fibers.

\begin{center}
\setlength{\unitlength}{10pt}
\begin{picture}(12,12)
\put (5,1){\makebox(1,1){$\mathcal{M}$}}
\put (1.25,5){\makebox(1,1){$M$}}
\put (8.75,5){\makebox(1,1){$\mathcal{Z}$}}
\put (5,9){\makebox(1,1){$C(M)$}}

\put (3,5.5){\vector(1,0){5}}
\put (4.5,8.5){\vector(-1,-1){2}}
\put (6.5,8.5){\vector(1,-1){2}}
\put (2.5,4.5){\vector(1,-1){2}}
\put (8.5,4.5){\vector(-1,-1){2}}

\put (2.5,7.8) {$\Sm{\R_+}$}
\put (8,7.8) {$\Sm{\C^*}$}
\put (5,6) {$\Sm{S^1}$}
\put (1.5,3) {$\Sm{\Sp(1)}$}\put (1.5,2.3) {$\Sm{\SO(3)}$}
\put (7.8,3) {$\Sm{\cps^1}$}

\end{picture}
\end{center}

The leaf space $\mathcal{Z}$ is an orbifold with complex contact structure, while $\mathcal{M}$ is a quaternionic-K\"{a}hler
orbifold.  This intimate relation with other more well known geometries is probably the reason 3-Sasaki manifolds
have not been studied as much quaternionic-K\"{a}hler manifolds.  For more details see~\cite{BoyGal99}

A 3-Sasaki manifolds $(M,g),\ \dim M=4m-1,$ is \emph{toric} if there is a $T^m \subseteq\Aut(M,g,\xi_1,\xi_2,\xi_3)$.
Toric 3-Sasaki manifolds have been constructed from 3-Sasaki quotients by torus actions on $S^{4n-1}$, with the 3-Sasaki structure given by right multiplication by $\Sp(1)$.  A subtorus $T^k \subset T^n$ is determined by a weight matrix $\Omega_{k,n}\in\operatorname{Mat}(k,n,\Z)$.  There are conditions on
$\Omega$~\cite{BGMR98}  that imply the moment map $\mu :S^{4n-1} \rightarrow(\mathfrak{t}^k)^*\otimes\R^3$
is a submersion, and further that the quotient
\[ M_{\Omega_{k,n}} = S^{4n-1}/\negthickspace/{T^k} =\mu^{-1}(0)/{T^k} \]
is smooth.

When $n=k+2$ it was shown in~\cite{BGMR98} that there are infinitely many weight matrices in $\operatorname{Mat}(k,n,\Z)$ for $k\geq 1$
giving infinitely many 7-manifolds $M_{\Omega_{k,n}}$ for each $b_2(M_{\Omega_{k,n}}) =k\geq 1$.

If $b_2(M)\geq 1$, then the maximal torus of \emph{Sasaki} automorphisms, $T^3 \subset\Aut(M,\xi_1)$, is 3-dimensional.
And if $b_2(M)\geq 2$, then the connected component of the identity of $\Isom(g) =T^2 \times\Sp(1)$ or $T^2 \times\SO(3)$,
where the second factor is generated by $\{\xi_1,\xi_2,\xi_3\}$.

\begin{prop}[\cite{vanC12}]
Let $(M,g)$ be a toric 3-Sasaki 7-manifold.  Then after fixing one of the Sasaki structures $(g,\eta_1,\xi_1,\Phi_1)$ with
foliation $\mathscr{F}_{\xi_1}$ we have
\[ H^1_{\ol{\partial}_b}(\mathcal{A}^{0,\bullet})=H^1_{\ol{\partial}_b}(\mathcal{A}^{0,\bullet})^{T^3} =b_2(M)-1=k-1.\]
\end{prop}

By Proposition~\ref{prop:Kuran-obst} there is a smooth Kuranishi space $\mathcal{B}$ of deformations of
$(\mathscr{F}_{\xi_1},\ol{J})$ equivariant with respect to $T^3$.  One can further show that the family $\mathcal{B}$ is
effective~\cite{vanC12}.  Let $\mathfrak{t}$ denote the Lie algebra of $T^3$.
By Corollary~\ref{cor:maincor-tor} there is a neighborhood $N\subset\mathcal{B}\times\mathfrak{t}$ parametrizing
a space of dimension $b_2(M) +2$ of Sasaki-extremal metrics.  And by Corollary~\ref{cor:maincor-Einst} there is a
$b_2(M)-1$-dimensional submanifold $U\subset N$ parametrizing a space of Sasaki-Einstein metrics.
Note that all these Sasaki-extremal structures satisfy Proposition~\ref{prop:CY-cond}, while only the submanifold $U\subset N$ of
Einstein metrics and their homotheties have constant scalar curvature since the rest have non-vanishing Futaki invariant.
See Figure~\ref{fig:def-space} which shows the isometry groups.  Note that the origin is 3-Sasaki while the other metrics
are just Sasaki-Einstein.  It is well known that 3-Sasaki structures are non-deformable, but these are the first examples
known to the author of 3-Sasaki structures contained in families of Sasaki-Einstein structures.  

Recall that a 3-Sasaki manifold $M$ with $\dim M=4m-1$ admits $m+1$ Killing spinors whereas a simply connected Sasaki-Einstein,
non-3-Sasaki, metric admits $2$.  So these families give examples of Einstein metrics admitting 3 Killing spinors with deformations
to Einstein metrics admitting only 2.  These properties are explored further in~\cite{vanC12}.

\begin{figure}[t!]
\labellist
\hair 2pt
\pinlabel $\C^{b_2 -1}$ at 241 234
\pinlabel $\R^{b_2-1}$ [l] at 278 142
\pinlabel $T^3$ at 113 181
\pinlabel {$T^3 \times\Z_2$} [r] at 40 109
\pinlabel {$T^2 \times\Sp(1)$} [l] at 173 145
\endlabellist
\centering
\includegraphics[scale=.5]{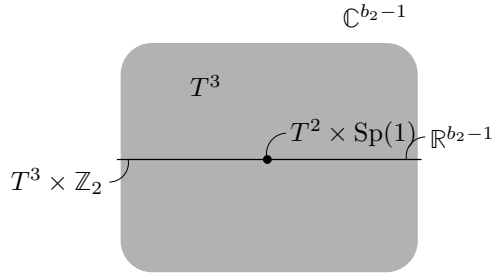}
\caption{Space of Sasaki-Einstein metrics}
\label{fig:def-space}
\end{figure}


\bibliographystyle{plain}

\end{document}